\documentclass[12pt]{article}

\usepackage[inner=32mm,outer=31mm,tmargin=30mm,bmargin=40mm]{geometry}
\usepackage[ngerman,english]{babel}

\usepackage{latexsym,amsfonts,amsmath,amssymb,epsfig,tabularx,amsthm,dsfont,mathrsfs}

\usepackage{graphicx}
\usepackage{enumerate}

\usepackage{booktabs}
\usepackage{titling}
\newcommand{\subtitle}[1]{%
  \posttitle{%
    \par\end{center}
    \begin{center}\large#1\end{center}
    \vskip0.5em}%
}

\usepackage{hyperref} 
\hypersetup{colorlinks=true,
        linkcolor=black,
        citecolor=black,
        urlcolor=blue}

\usepackage{pgfplots}

\usepackage{framed}
\usepackage{amscd}

\usepackage{tikz-cd}
\usepackage{tikz}

\theoremstyle{plain}

\newtheorem{theorem}{Theorem}[section]
\newtheorem{lemma}[theorem]{Lemma}

\newtheorem{proposition}[theorem]{Proposition}
\newtheorem{corollary}[theorem]{Corollary}

\newtheorem{consequence}{Consequence}

\theoremstyle{definition}
\newtheorem{remark}[theorem]{Remark}

\renewcommand{\P}{{\mathbb P}}
\newcommand{\expect}{\operatorname{\mathbb{E}}}

\DeclareMathOperator{\Uniform}{unif}
\DeclareMathOperator{\Normal}{\mathcal{N}}

\newcommand{\iid}{\stackrel{\textup{iid}}{\sim}}

\newcommand{\dint}{\,\mathup{d}}

\DeclareMathOperator{\linspan}{span}

\newcommand{\ind}{\mathds{1}}
\DeclareMathOperator*{\argmin}{argmin}

\DeclareMathOperator{\dist}{dist}

\DeclareMathOperator{\id}{id}

\DeclareMathOperator{\diag}{diag}

\newcommand{\ran}{\textup{ran}}
\newcommand{\deter}{\textup{det}}

\newcommand{\nonada}{\textup{nonada}}
\newcommand{\lin}{\textup{lin}}

\newcommand{\eps}{\varepsilon}
\newcommand{\euler}{e} 
\newcommand{\embed}{\hookrightarrow}

\renewcommand{\vec}{\boldsymbol}

\newcommand{\R}{{\mathbb R}}

\newcommand{\N}{{\mathbb N}}

\DeclareMathAlphabet{\mathup}{OT1}{\familydefault}{m}{n}

\newcommand{\wt}{\widetilde}
\newcommand{\widebar}[1]{\mbox{\kern1.5pt\hbox{\vbox{\hrule height 0.6pt \kern0.35ex
        \hbox{\kern-0.15em \ensuremath{#1 }\kern0.0em}}}}\kern-0.1pt}

\newlength{\fixboxwidth}
\usepackage{soul}

\usepackage{latexsym,amsfonts,amsmath,amssymb,mathrsfs}
\usepackage{url}
\usepackage{graphicx}
\usepackage{color}
\usepackage{euscript}
\usepackage{verbatim}
\usepackage{hyperref}

\begin{document}



\title{Randomized approximation of summable sequences -- adaptive and non-adaptive}

\author{Robert J. Kunsch\thanks{RWTH Aachen University,
at the Chair of Mathematics of Information Processing,
Pontdriesch~10,
52062 Aachen, Email: kunsch@mathc.rwth-aachen.de
}, 
Erich Novak\thanks{FSU Jena, 
Ernst-Abbe-Platz 2, 03641 Jena, Email: 
erich.novak@uni-jena.de}, 
Marcin Wnuk\thanks{Institut für Mathematik, 
Osnabrück University, Albrechtstraße 28a, 49076 Osnabrück, 
Email: marcin.wnuk@uni-osnabrueck.de, corresponding author}}

\date{\today}

\maketitle
\begin{abstract}
We prove lower bounds for the randomized approximation
of the embedding $\ell_1^m \embed \ell_\infty^m$
based on algorithms that use arbitrary linear 
(hence non-adaptive) information
provided by a (randomized) measurement 
matrix $N \in \R^{n \times m}$.
These lower bounds 
reflect the increasing difficulty of the 
problem for $m \to \infty$,
namely, a term 
$\sqrt{\log m}$ in the complexity $n$.
This result implies that non-compact operators 
between arbitrary Banach spaces
are not approximable using non-adaptive Monte Carlo methods. 
We also compare these lower bounds
for non-adaptive methods 
with upper bounds based on adaptive, randomized 
methods for recovery
for which the complexity $n$ only 
exhibits a $(\log\log m)$-dependence. 
In doing so we give an example of linear problems where 
the error for adaptive vs.\ non-adaptive Monte Carlo methods
shows a gap of order $n^{1/2} ( \log n)^{-1/2}$.
\end{abstract}



\

\section{Introduction}
Consider
a linear problem given by
a linear operator
$
  S \colon\; F \to G
$
between Banach spaces $F$ and $G$. We are interested in the randomized approximation error, namely for a randomized algorithm $A_n$ and an input $f \in F$, the output $A_n(f)$ is a $G$-valued random variable and we define its error by
\begin{equation*}
    e(A_n,S,f) := \expect \|A_n(f) - S(f)\|_G \,.
\end{equation*}
A randomized algorithm can be seen as a family $A_n = (A_n^\omega)_{\omega \in \Omega} = (\phi^\omega \circ N^\omega)_{\omega \in \Omega}$
of mappings $F \to G$
on a complete probability space $(\Omega, \mathcal{A},\P)$
where each realization $A_n^\omega$ can be decomposed
into an information map $N^\omega \colon F \to \R^n$
that uses $n$ linear functionals,
and a reconstruction map $\phi^\omega \colon \R^n \to G$.
In the simplest setting of \emph{non-adaptive} information,
$N^\omega$ is a linear map yielding information
\begin{equation*}
    \vec{y} = (y_1,\ldots,y_n) = N^\omega(f)
    = \bigl(L_1^\omega(f), \ldots, L_n^\omega(f)\bigr)
\end{equation*}
with linear functionals $L_1^\omega,\ldots,L_n^\omega \in F^\prime$.
In a finite-dimensional problem setup with $F = \R^m$,
a non-adaptive information map $N = (N^\omega)_{\omega \in \Omega}$
can be identified with a random $(n \times m)$-matrix.
The more general case of \emph{adaptive} information
allows $L_2^\omega,\ldots,L_n^\omega$ to be chosen depending on previous
information, in detail, for $k=2,\ldots,n$ we may have
\begin{equation*}
    y_k := L_k^\omega(f;y_1,\ldots,y_{k-1})
    \qquad\text{where}\quad L_k^\omega(\,\cdot\,;y_1,\ldots,y_{k-1}) \in F^\prime.
\end{equation*}
We assume Borel measurability for the maps $\Omega \times \R^{k-1} \to F'$ selecting the $k$-th information functionals, 
thus guaranteeing the measurability of the information map
$N \colon \Omega \times F \to \R^n$
as a whole.
We also assume measurability of the reconstruction
$\phi\colon \Omega \times \R^n \to G$
which implies measurability of the algorithm $A_n \colon \Omega \times F \to G$.

%
We compare algorithms that use (at most) $n$ linear functionals 
and call $n$ the (information) cost of such an algorithm.
This leads to the following notion of a minimal Monte Carlo error for the problem $S$,
\begin{equation*}
  e^{\ran}(n,S) := \inf_{(A_n^\omega)_{\omega \in \Omega}} \sup_{\|f\|_F \leq 1} \expect \|A_n(f) - S(f)\|_G \,;
\end{equation*}
the minimal $n$ needed to guarantee an error at most $\eps$ is called complexity. 
We thus use the standard worst case setting of 
approximation theory and information-based complexity
for randomized algorithms.
In this paper, we also use the term \emph{Monte Carlo method}
to refer to any randomized algorithm.

If $A_n^\omega \colon F \to G$ does not depend on $\omega$, the algorithm is called \emph{deterministic} and we write $A_n \colon F \to G$. The corresponding error notion is
\begin{equation*}
  e^{\deter}(n,S) := \inf_{A_n} \sup_{\|f\|_F \leq 1} \|A_n(f) - S(f)\|_G \, ,
\end{equation*}
where the infimum is taken with respect to all 
deterministic algorithms with information cost at most $n$. 
An algorithm is called non-adaptive
if its information mapping $N$ is non-adaptive,
otherwise it is called adaptive.
We define 
\begin{equation*}
  e^{\ran, \nonada}(n,S) := \inf_{(A_n^\omega)_{\omega \in \Omega}} \sup_{\|f\|_F \leq 1} \expect \|A_n(f) - S(f)\|_G 
\end{equation*}
as above, but the infimum is taken over the set of 
non-adaptive algorithms.
By writing $e^{\ran,\lin}$ we indicate the restriction to algorithms $A_n$
based on linear maps $\phi^\omega$ and $N^\omega$.
(A slight generalization of these error notions allows for algorithms with \emph{varying cardinality} where the amount of information $n$ may be random or even adaptively chosen and is only bounded in terms of its expected value.
It turns out, that, at least in the non-adaptive setting, varying cardinality $n(\omega)$ cannot improve much over the setting with strict cost bound,
see~\cite[pp.~289/290]{H92} or \cite[Lem~1.4]{Ku17}.)

For the main part of this paper,
we study 
the identical embedding between finite dimensional sequence spaces,
\begin{equation*}
  S \colon \ell_p^m \to \ell_q^m,\quad  \vec{x} \mapsto \vec{x} \,,
\end{equation*}
where $1 \leq p,q \leq \infty$ and $m \in \N$.
We will write $\ell_p^m \embed \ell_q^m$ as a shorthand for 
the sequence space embedding $S$,
and $[m] := \{1,\ldots,m\}$ for the index set of a vector $\vec{x} = (x_1,\ldots,x_m) \in \R^m$.
We consider approximation with incomplete information, namely $n < m$, say, $m \geq 4n$.
A particular interest lies in extreme cases with $n \ll m$,
or in other words, the asymptotic behaviour for $\frac{n}{m} \to 0$,
or even just $m \to \infty$ for fixed $n$.
Sequence space embeddings play an 
important role in understanding more general approximation problems, for instance contrasting approximation rates for different classes of algorithms: Linear vs.\ non-linear, adaptive vs. non-adaptive, deterministic vs. randomized.
We quickly review classical results for three important special cases,
for a comprehensive overview of deterministic approximation theory in sequence spaces, see~\cite[Chap~VI]{Pin85} and \cite[Chap~10]{FR2013CS}.

\begin{description}
    \item[$\ell_1^m \embed \ell_2^m$:] From 
      Kashin~(1974)~\cite{Ka74} and Garnaev, Gluskin (1984)~\cite{GG84} 
      we know deterministic rates,
      namely, for $m \geq 2n$ we have
      \begin{equation} \label{eq:KGG12}
        e^{\deter}(n,\ell_1^m \embed \ell_2^m)
          \asymp \min\left\{1,\, \sqrt{\frac{\log\frac{m}{n}}{n}}\right\} \,.
      \end{equation}
      Upper bounds in this case are achieved by using
      a non-linear reconstruction $\phi(\vec{y}) := \argmin\limits_{\vec{z} \in \R^m\colon N\vec{z} = \vec{y}} \|\vec{z}\|_1$
      based on suitable non-adaptive information \mbox{$\vec{y} = N\vec{x}$}.
      Non-linearity, indeed, is a necessary feature as lower bounds for linear randomized methods show: Math\'e (1991)~\cite[Sec~3, Lem~2 (ii)]{Ma91} implies that for $m \geq 2n$,
      \begin{equation*}
        e^{\ran,\lin}(n,\ell_1^m \embed \ell_2^m)
          \asymp e^{\deter,\lin}(n,\ell_1^m \embed \ell_2^m)
          \asymp 1 \,.
      \end{equation*}
      From Heinrich (1992)~\cite{H92} we know lower bounds for general randomized methods, namely,
      for $m \geq 2n$, he proves 
      \begin{equation} \label{eq:He12}
        e^{\ran}(n,\ell_1^m \embed \ell_2^m) \succeq n^{-1/2} \,.
      \end{equation}
      This bound shows that randomized methods cannot improve the $n$-dependence
      over deterministic algorithms,
      but it does not reflect an $m$-dependence.
    \item[$\ell_2^m \embed \ell_\infty^m$:]
      Here, in the deterministic setting, for $m \geq 2n$, no substantial approximation is possible,
      \begin{equation*}
        e^{\deter}(n,\ell_2^m \embed \ell_\infty^m) \asymp 1 \,,
      \end{equation*}
      while using linear (thus non-adaptive) randomized 
      algorithms, Math\'e ~\cite{Ma91} shows 
      that for $m \geq 2n$,
      \begin{equation*}
        e^{\ran}(n,\ell_2^{m} \embed \ell_\infty^{m})
        \preceq  \sqrt{\frac{\log m}{n}}\,.
      \end{equation*}
      The algorithms used for this result
      consist of a linear information map representable
      as $(n\times m)$-matrix $N$ with i.i.d.\ standard Gaussian entries,
      and a linear reconstruction map representable as the matrix $\varphi^\omega = \frac{1}{n} (N^\omega)^\top \in \R^{m \times n}$.
      Once again, from Heinrich~\cite{H92} we know a lower bound for $m \geq 2n$
      without an $m$-dependence
      \begin{equation} \label{eq:He2oo}
        e^{\ran}(n,\ell_2^m \embed \ell_\infty^m) \succeq \sqrt{\frac{\log n}{n}} \,.
      \end{equation}
    \item[$\ell_1^m \embed \ell_\infty^m$:]
      In the deterministic setting there exist linear methods that are optimal,
      see \cite{CW04} or \cite[Thms~4.5 and~4.8]{NW08},
      yet the precise rates remain unknown.
      The rates from non-linear 
      deterministic methods in the setting $\ell_1^m \embed \ell_2^m$ with the stricter error criterion, see~\eqref{eq:KGG12},
      directly translate into upper bounds for $\ell_1^m \embed \ell_\infty^m$.
      Lower bounds reflecting a logarithmic $m$-dependence can be found in~\cite[Prop~2.1]{FPRU10} and the rate in $n$ is also known to be optimal.
      Math\'e~\cite[Sec~3, Lem~2 (ii)]{Ma91} showed
      for $m \ge 2n$ that
      \begin{equation*}
        e^{\ran,\lin}(n,\ell_1^m \embed \ell_\infty^m) \succeq n^{-1/2} \,,
      \end{equation*}
      hence, linear randomized methods cannot 
      improve the rate in $n$ over deterministic methods,
      yet the optimal $m$-dependence remains unknown.
      Combining the non-linearity for $\ell_1^m \embed \ell_2^m$ and the randomness for $\ell_2^m \embed \ell_\infty^m$, 
      Heinrich~\cite{H92} constructed a two-stage 
      algorithm showing an improved main rate (for $m \geq 4n$), together with a lower bound:
      \begin{equation} \label{eq:He1oo}
        \frac{\sqrt{\log n}}{n}
          \preceq e^{\ran}(n,\ell_1^{m} \embed \ell_\infty^{m})
          \preceq  \frac{\sqrt{(\log \frac{m}{n})(\log m)}}{n} \,.
        \end{equation}
      Here as well, the lower bound lacks an $m$-dependence.
\end{description}

Our main contribution is a lower bound for non-adaptive randomized approximation that reflects the $m$-dependence. 
Theorem~\ref{thm:l1->loo,non-ada} shows
for $n \leq c_0 \sqrt{\log m}$ (with a suitable constant $c_0 > 0$) that
\begin{equation*}
  e^{\ran,\nonada}(n,\ell_1^m \embed \ell_\infty^m) \geq \eps_0 > 0 \,.
\end{equation*}
This result has two main consequences which we state and comment below.

\begin{consequence}
    Non-compact operators 
    between arbitrary Banach spaces
    are not approximable using non-adaptive Monte Carlo methods,
    see Theorem~\ref{Thm:Approxmability}.
\end{consequence}

This statement relies on the fact
that the sequence space embedding $\ell_1 \embed \ell_\infty$
is known to be a universal non-compact operator~\cite{J71},
see Proposition~\ref{prop:UniversalNonComp}.
The 
first consequence therefore justifies our focus on lower bounds for the problem $\ell_1^m \embed \ell_\infty^m$.
On an elementary level, among all sequence space embeddings $\ell_p^m \embed \ell_q^m$ with $1 \leq p,q \leq \infty$,
the case $(p,q) = (1,\infty)$ is the easiest from an approximation point of view:
The $\ell_1$-unit ball $B_1^m$ is the smallest among all $\ell_p$-balls,
further, the $\ell_\infty$-norm is the least punishing among all $\ell_q$-norms we might choose to measure the error.
Hence, our lower bound holds for all embeddings
from $\ell_p^m$ to $\ell_q^m$.

Lower bounds are also studied in the field of sparse 
recovery with a different 
error criterion,
see in particular Woodroff et al.~\cite{BIPW10,PW12}
and also \eqref{eq:crite} in Section~\ref{sec:UB}.
The lower bounds in~\cite{PW12} can be transferred to lower bounds for $\ell_2^m \embed \ell_\infty^m$
and cover both the adaptive and the non-adaptive setting.
Our proof technique as described in 
Section~\ref{sec:bakhvalov} is a substantial modification of their approach suited to accommodate the much smaller input set $B_1^m$ in the problem $\ell_1^m \embed \ell_\infty^m$.
However,
we only discuss the non-adaptive case, the adaptive case is left for future research.
The setting in~\cite{BIPW10} resembles our setting with input space $\ell_1^m$ and contains a lower bound for non-adaptive algorithms.
Their proof, relying on communication complexity, 
seems not to
translate easily to
our error criterion, though.
On the other hand, upper bounds in the sparse recovery setting~\cite{IPW11,LNW17}
can easily be used to prove upper bounds in our setting, see Section~\ref{sec:UB}
where we find upper bounds using adaptive algorithms which lead to the second main consequence.

\begin{consequence}
    Adaptive Monte Carlo methods are much better 
    than non-adaptive methods for some linear problems. 
    The gap in the error can be of the order 
    $n^{1/2} (\log n)^{-1/2}$
    for the embedding 
     $\ell_1^m \embed \ell_2^m$,
       see Corollary~\ref{cor:adagap}. 
\end{consequence}

It is well known that for linear problems in the deterministic setting,
adaption can reduce the error by a factor $2$ at most,
hence, no difference in the approximation rates can be observed,
see the survey \cite{No96} for this and for related results. 
It was not known whether such a result also holds for randomized 
algorithms. 
The problem was posed in \cite{No96} and
restated in \cite[Open Problem 20]{NW08}. 

This open problem was recently solved by Stefan Heinrich~\cite{He22,He23b,He23,He24}  
who studied (parametric) integration and approximation in mixed $\ell_p(\ell_q)$-spaces using standard information (function evaluations).
In~\cite{He23} he proved 
that in the randomized setting a gap of order $n^{1/2} ( \log n)^{-1}$
is possible between adaptive and non-adaptive approximation.
Our gap is slightly larger in the logarithmic term and
one can prove that the largest possible gap is 
at most $n^{1/2}$ for mappings into a Hilbert space, 
see Remark~\ref{rem:maximalgap}. 
Note that our result is obtained for the class of all 
linear functionals as admissible information rather than 
just function evaluations,
hence, in a 
different algorithmic setting which requires different proof techniques.

The problem $\ell_1^m \embed \ell_\infty^m$ is notable due to the fact
that in the deterministic setting, additional algorithmic features like non-linearity
or adaptivity cannot improve the error whatsoever,
but in the randomized setting, the algorithmic features of non-linearity and adaptivity
both bring about significant improvements.

\subsection*{Notation}
For a normed space $(F,\|\cdot\|_F)$ we denote its closed unit ball by $B_F$,
the distance between two sets $F_1, F_2 \subset F$ is defined via
\begin{equation*}
  \dist_F(F_1, F_2) := \inf_{\substack{f_1 \in F_1 \\ f_2 \in F_2}} \|f_1 - f_2\|_F \,.
\end{equation*}
In spaces $\ell_p^m$ we use the abbreviations $\|\cdot\|_p := \|\cdot\|_{\ell_p^m}$ for the norm,
$B_p^m := B_{\ell_p^m}$ for the unit ball, and $\dist_p := \dist_{\ell_p^m}$ for the distance.
When talking about operators between normed spaces, we always mean bounded linear operators.

For functions $f,g \colon M \to (0,\infty)$ on a domain $M$ we write $f \preceq g$
if there exists a constant $C > 0$ such that for all $x \in M$ we have $f(x) \leq C g(x)$.
If $f \preceq g$ and $f \succeq g$, we write $f \asymp g$
and call $f$ and $g$ \emph{weakly asymptotically equivalent}.

Throughout this paper we work with an underlying probability space $(\Omega,\mathcal{A},\P)$.
For $\Omega' \subset \Omega$ we denote its complement by $\Omega'^c := \Omega \setminus \Omega'$.
Furthermore, all the measurable spaces are tacitly assumed to be Polish, so that one can use the disintegration theorem to produce conditional measures.
We use the following conventions concerning conditional probabilities
with respect to a random variable $\vec{Y}$ on a measurable space $H$:
We write $\P^{\vec{Y}}(\,\cdot\,) := \P(\vec{Y} \in \cdot\,)$ for the distribution of $\vec{Y}$,
and let the mapping $H \times \mathcal{A} \to [0,1],\, 
(\vec{y},A) \mapsto \P_{\vec{y}}(A) = \P(A \mid \vec{Y} = \vec{y})$
denote a regular conditional probability
(for details see, e.g., \cite[Chap~8]{Kl14})
such that
\begin{equation*}
  \P(\,\cdot\,) = \int_{H} \P_{\vec{y}}(\,\cdot\,) \dint\P^{\vec{Y}}(\vec{y}) \,.
\end{equation*}
Accordingly, we write $\expect_{\vec{y}}[\;\cdot\;] = \expect[\;\cdot\mid \vec{Y} = \vec{y}]$
for integration with respect to $\P_{\vec{y}}$.

\section{A lower bound for non-adaptive algorithms}

\subsection{The average-case setting}
\label{sec:bakhvalov}

We make use of Bakhvalov's technique for Monte Carlo lower bounds~\cite{Bakh59,H92,Ma93}.
Namely, for a (sub-)probability measure $\mu$ supported on the unit ball $B_F$ of the normed space~$F$ (that is, $\mu \geq 0$ and $\mu(F) = \mu(B_F) \leq 1$), consider the \emph{average case error}
\begin{equation*}
  e^\mu(n,S) := \inf_{A_n} \int \|A_n(f) - S(f)\|_G \dint \mu(f) \, ,
\end{equation*}
where the infimum is taken with respect to all measurable deterministic algorithms using at most $n$ information functionals.
Exploiting measurability of randomized algorithms as mappings $\Omega \times F \to G$,
we have the lower bound
\begin{equation} \label{eq:Bkahvalov}
  e^{\ran}(n,S) \geq e^\mu(n,S) \,.
\end{equation}
Restricting to non-adaptive deterministic algorithms in the $\mu$-average setting
yields lower bounds for non-adaptive randomized algorithms.
The lower bound results cited in the introduction
rely on truncated Gaussian measures on the unit ball $B_p^m \subset \ell_p^m$
intersected with a subspace, see Heinrich~\cite{H92}.
The problem with scaled $m$-dimensional standard Gaussian measures
is that for $p < q$ already the \emph{initial error}
of their respective truncations $\mu$ on the unit ball $B_p^m$ decays to zero,
no matter how we change the scaling:
\begin{equation*}
  e^\mu(0,\ell_p^m \embed \ell_q^m)
    = \int \|\vec{x}\|_{q} \dint \mu(\vec{x})
    \xrightarrow[m \to \infty]{} 0 \,.
\end{equation*}
For that reason, Heinrich~\cite{H92} only considered $4n$-dimensional Gaussian measures (if $m \geq 4n$) resulting in lower bounds that do not reflect the increasing difficulty of the problem with growing $m$.

In what follows, we will consider a random input of the form
\begin{equation} \label{eq:X=u+sW}
  \vec{X} := \vec{u}_I + \sigma\vec{W}\,.
\end{equation}
Here, $(\vec{u}_i)_{i \in [M]}$, $M \in \N$, is a collection of
points from the $\ell_p^m$-unit ball $B_p^m \subset \R^m$ (we are especially interested in the $\ell_1$-ball $B_1^m$),
and $I \sim \Uniform [M]$
is a selector random variable which chooses the random point $\vec{u}_I$.
The disturbance $\vec{W}$ is generated by projecting a standard Gaussian vector $\vec{Z} \sim \Normal(\vec{0},\id_m)$ onto $2n$ randomly chosen coordinates
and we introduce a scaling $\sigma > 0$.
In detail, we define the set of subsets
\begin{equation}\label{eq:setsJn'}
  \mathcal{J}_{2n}
    := \left\{\mathfrak{j} \subseteq [m] \colon \#\mathfrak{j} = 2n\right\} \,,
\end{equation}
and write $P_{\mathfrak{j}} \colon \R^m \to \R^m$ for the projection onto the coordinates given by an index set $\mathfrak{j} \subseteq [m]$.
Taking $\mathfrak{J} \sim \Uniform{\mathcal{J}_{2n}}$, we define $\vec{W} := P_{\mathfrak{J}} \vec{Z}$, i.e.
\begin{equation*}
  W_j = \begin{cases}
          Z_j &\text{if } j \in \mathfrak{J} \,, \\
          0 & \text{if } j \notin \mathfrak{J} \,.
        \end{cases}
\end{equation*}
Here, $I$, $\vec{Z}$, and $\mathfrak{J}$ are independent random variables.
The distribution of $\vec{X}$ is a measure
\begin{equation}
  \wt{\mu} := \P^{\vec{X}}
\end{equation}
on $\R^m$. We assume that, with suitably chosen radius $r_0 > 0$, the $\ell_{p}^m$-balls
\begin{equation}\label{eq:Fi_general}
  F_i := \vec{u}_i + r_0 B_p^m
\end{equation}
around each point $\vec{u}_i$ fit inside the unit ball $B_p^m$ and are separated by $c > 0$ in the $\ell_q^m$-metric,
\begin{equation*}
  F_i \subseteq B_p^m \qquad\text{and}\qquad
  \dist_{q}(F_i,F_{i'}) \geq c \quad\text{for } i \neq i'.
\end{equation*}
We truncate the measure $\wt{\mu}$ with respect to the following event,
\begin{equation} \label{eq:Omega'}
  \Omega' := \left\{\omega \in \Omega \colon \vec{X}(\omega) \in F_{I(\omega)}\right\}
    = \left\{\omega \in \Omega \colon \|\vec{W}(\omega)\|_p \leq \frac{r_0}{\sigma}\right\} \,,
\end{equation}
namely, we define the truncated measure
\begin{equation}\label{eq:mu-general-truncation}
  \mu(\,\cdot\,) := \P\bigl(\Omega' \cap \{\vec{X} \in \,\cdot\,\}\bigr) \,.
\end{equation}
That way we ensure in particular that $\mu$ is supported on~$B_p^m$
and we can indeed apply Bakhvalov's technique~\eqref{eq:Bkahvalov}.
The scaling $\sigma > 0$ of the disturbance
has an influence on the truncation probability
\begin{equation}
  \delta := \wt{\mu}(\R^m) - \mu(\R^m) = \P\left(\Omega'^c\right) \,.
\end{equation}
Integrating with respect to the truncated measure $\mu$ is equivalent to taking the
expectation multiplied with the indicator function of $\Omega'$,
that is, the corresponding $\mu$-average error of a deterministic (measurable) algorithm $A_n \colon \R^m \to \R^m$ is
\begin{equation*}
  \int \|A_n(\vec{x}) - \vec{x}\|_q \dint\mu(\vec{x})
    = \expect \left[ \|A_n(\vec{X}) - \vec{X}\|_q \cdot \ind_{\Omega'} \right] \,.
\end{equation*}
In the non-adaptive setting, $A_n = \phi \circ N$ is built with a linear information mapping $N \colon \R^m \to \R^n$ which can be identified with a matrix $N \in \R^{n \times m}$.
By subsequent Lemma~\ref{lem:cond-e-avg},
finding lower bounds on the error of the average case problem
reduces to finding lower bounds on the failure probability for identifying $I$
from measurements $\vec{Y} := N\vec{X}$.

Measures similar to the distribution of $\vec{X}$ in \eqref{eq:X=u+sW} 
have been considered in \cite{PW12} in a compressed sensing model with $\ell_2$-noise.
If the source space is $\ell_2^m$, then instead of $2n$-dimensional disturbances $\sigma \vec{W}$ we may directly work with the $m$-dimensional Gaussian vector $\sigma\vec{Z}$ and omit the projection step.
The reason for the projection step 
is that,
in order to keep the truncation probability~$\delta$ small,
in the $\ell_1$-setting the scaling factor $\sigma$ would need to shrink prohibitively fast with growing dimension of the Gaussian measure.
In the $\ell_2$-setting, however, the projection step is not necessary;
the analysis of adaptive algorithms then is way easier and has indeed been conducted in~\cite{PW12}.

For the analysis of the average case setting with different directions $\mathfrak{j} \subset [m]$
of the disturbance $\vec{W}$,
 we need a monotonicity argument for information.

\begin{lemma}\label{lem:monotonInfo}
  Let $S \colon F \to G$ be a measurable map between Banach spaces and
  let $\vec{X}$ be a random variable with values in $F$.
  Consider random variables $\vec{Y}$ and $\wt{\vec{Y}}$ with values in measurable spaces $H$ and $\wt{H}$, respectively,
  so-called \emph{information},
  and assume that
  $\vec{Y}$ can be written as a measurable function of $\wt{\vec{Y}}$.

  Then, for any event $\Omega' \subset \Omega$, we have
  \begin{equation*}
    \inf_{\wt{\phi}\colon \wt{H} \to G}
        \expect \left[\|S(\vec{X}) - \wt{\phi}(\wt{\vec{Y}})\|_G \cdot \ind_{\Omega'}\right]
      \leq \inf_{\phi \colon H \to G} 
        \expect \left[\|S(\vec{X}) - \phi(\vec{Y})\|_G \cdot \ind_{\Omega'}\right] \,,
  \end{equation*}
  where the infima are taken over measurable mappings.
\end{lemma}
\begin{proof}
    Let $f:\wt{H} \rightarrow H$ be a measurable mapping 
    with $\vec{Y} = f(\wt{\vec{Y}})$ almost surely.
    For each measurable $\phi \colon H \to G$, the mapping $\phi \circ \vec{Y}$ 
    can be replicated via a measurable mapping $\phi \circ f \circ \wt{\vec{Y}}$, 
    that is, $\wt{\phi} = \phi \circ f$.
\end{proof}

As described beforehand, we will analyse the average case error $e^\mu(n,\ell_p^m \embed \ell_q^m)$ 
for which we only need to consider deterministic algorithms $A_n := \phi \circ N$ 
with information map $N \colon \ell_p^m \to \R^n$ and reconstruction map $\phi \colon \R^n \to \ell_q^m$.
There can only be one output $\phi(\vec{y}) \in \ell_q^m$ for one information vector $\vec{y} \in \R^n$.
This output shall be optimal with respect the truncated conditional distribution.
The following lemma provides an error bound for the conditional average error,
in later applications we take $\P_y$ for $\P$ and $\expect_y$ for $\expect$.

\begin{lemma} \label{lem:cond-e-avg}
  Let $S: F \to G$ be a measurable map between Banach spaces, and
  consider a family of sets $(F_i)_{i = 1}^M$ in $F$
  where there is
  a constant $c > 0$ such that
  \begin{equation*}
    \dist_G(S(F_i),S(F_{i'})) \geq c
  \end{equation*}
  for distinct indices $i,i' \in [M]$. Put $F' := \bigcup_{i=1}^M F_i$.
Furthermore, let $\vec{X}$ be a random variable with values in $F$
and consider an event
$\Omega' \subset \{ \vec{X} \in F'\}$. Then
\begin{equation*}
    \inf_{g \in G} \mathbb{E}\left[\lVert S(\vec{X}) - g  \rVert_G \cdot \ind_{\Omega'}\right] \geq c \min \left\{  \frac{\P(\Omega')}{2}, \P(\Omega') - \max_{i \in [M]} \P(\Omega' \cap \{\vec{X} \in F_i\}) \right\} \,.
\end{equation*}
\end{lemma}
\begin{proof}
  Fix $g \in G$ and assume that $d := \dist_G(g, S(F_M)) \leq \dist_{G}(g,S(F_i))$
  for all $i \in [M-1]$.
  We write 
  $t := \P(\Omega' \cap \{ \vec{X} \in F_M\})$ 
  and 
  $T := \P(\Omega')$.
  From the triangle inequality, for $i \in [M-1]$, we have
  \begin{align*}
    c &\leq \dist_G(S(F_i), S(F_{M})) \leq \dist_G(S(F_i), g) + \underbrace{\dist_G(g, S(F_M))}_{= d} \\
    &\quad\Longrightarrow\quad 
    \dist_G(S(F_i), g) \geq c - d\,.
  \end{align*}
  This leads to
  \begin{align*}
 \mathbb{E}\left[\lVert S(\vec{X}) - g  \rVert_G \cdot \ind_{\Omega'}\right] 
 &= \sum_{i = 1}^M\mathbb{E}\left[ \lVert S(X) - g \rVert_G \cdot \ind_{\Omega'}\cdot \ind_{\{\vec{X} \in F_i\}} \right]
 \\
 & \geq \sum_{i = 1}^M \P\bigl(\Omega' \cap \{\vec{X} \in F_i\}\bigr) \cdot \dist_G(S(F_i), g)
 \\
 & \geq \sum_{i = 1}^{M-1} \P\bigl(\Omega' \cap \{\vec{X} \in F_i\}\bigr) \cdot \max(d,c-d) + td.
  \end{align*}
  The last expression above is equal to $h_t(d)$ where
  \begin{align*}
   h_t: (0, \infty) \rightarrow \R, \quad 
     & d \mapsto \begin{cases}
  (2t - T)d +Tc - tc & \text{ if }  d \leq \frac{c}{2} \,,
  \\
   Td  & \text{ if }  d > \frac{c}{2} \,.
  \end{cases}
  \end{align*}
  
  The minimum of $h_t(\cdot)$ is to be found for $d \in [0, \frac{c}{2}]$, namely, due to the linearity, in one of the endpoints:
  \begin{equation*}
    h_t(0) = (T-t) \cdot c \,, \qquad
    h_t\left(\frac{c}{2}\right) = T \cdot \frac{c}{2} \,.
  \end{equation*}
  This shows
  \begin{equation*}
    \mathbb{E}\left[ \lVert S(X) - g  \rVert_G \cdot \ind_{\Omega'}  \right]
        \geq c \cdot \min\left\{\frac{T}{2},\, T-t\right\} \,,
  \end{equation*}
  where  $t = \P(\Omega' \cap \{ \vec{X} \in F_M\})$ 
  and 
  $T = \P(\Omega')$.
  Ending up with $\P(\Omega' \cap \{ \vec{X} \in F_M\})$ in this formula, however, was only possible under the simplifying assumption 
  that the output $g$ is closest to $S(F_M)$.
  Replacing $t$ by $\max_{i \in [M]}\P(\Omega' \cap \{\vec{X} \in F_i\})$, 
  we account for the possiblity of~$g$ being closer to another set $S(F_i)$.
\end{proof}

Recall that for the truncated measure $\mu$,
knowing the index $I=i$ means knowing in which of the balls $F_i = \vec{u}_i + r_0 B_p^m$
the input $\vec{X}$ lies.

\begin{lemma} \label{lem:truncation}
  Let $S: F \to G$ be a measurable map between Banach spaces.
    Consider a family of sets $(F_i)_{i = 1}^M$ in $F$ where there is a constant $c > 0$ such that
  \begin{equation*}
    \dist_G(S(F_i),S(F_{i'})) \geq c
  \end{equation*}
  for distinct indices $i,i' \in [M]$.

  Let $\vec{X}$ be a random variable with values in $F$,
  let $I$ be a random variable with values in $[M]$, and consider the event
  \begin{equation*}
    \Omega' = \left\{\omega \in \Omega \colon \vec{X}(\omega) \in F_{I(\omega)}\right\}
  \end{equation*}
  with $\delta := \P(\Omega'^c)$.
  Furthermore, let $\vec{Y}$ be a random variable with values in some measurable space $H$,
  the so-called \emph{information}.
  We define the ``distinctness'' function
  \begin{equation*}
    D_{I\mid\vec{Y}}(\vec{y}) := \max_{i \in [M]} \P(I = i \mid \vec{Y} = \vec{y}) \,.
  \end{equation*}
  
  Then for any measurable reconstruction map $\phi \colon H \to G$
  we have the average case lower bound
  \begin{equation*}
    \expect \bigl[ \|S(\vec{X}) - \phi(\vec{Y})\|_G \cdot \ind_{\Omega'}\bigr]
      \geq c \cdot \left(\frac{1}{2} \cdot
                    \P\left(D_{I\mid\vec{Y}}(\vec{Y}) \leq \frac{1}{2}\right)
                - \delta\right) \,.
  \end{equation*}
\end{lemma}

\begin{proof}
The case $\P(\Omega') = 0$ is trivial, so let us assume that $\P(\Omega') > 0$.
Defining the loss function
\begin{equation*}
  \ell(\vec{y})
    := \expect\bigl[\|S(\vec{X}) - \phi(\vec{Y})\|_G \cdot \ind_{\Omega'}
                    \;\big|\; \vec{Y} = \vec{y}\bigr] \,,
\end{equation*}
we can write
\begin{equation*}
  \expect \bigl[ \|S(\vec{X}) - \phi(\vec{Y})\|_G \cdot \ind_{\Omega'}\bigr] 
  = \expect\bigl[\ell(\vec{Y})\bigr] \,.
\end{equation*}
Using Lemma \ref{lem:cond-e-avg},
we see that for almost every $\vec{y}$,
\begin{align*}
\ell(\vec{y})
  &\geq c \cdot \min \left\{  \frac{\P_{\vec{y}}(\Omega')}{2}, \P_{\vec{y}}(\Omega') - \max_{i \in [M]} \P_{\vec{y}}( \Omega' \cap \{\vec{X} \in F_i\}) \right\} \nonumber\\
  &\geq c \cdot \min \left\{  \frac{\P_{\vec{y}}(\Omega')}{2}, \P_{\vec{y}}(\Omega') - \max_{i \in [M]} \P_{\vec{y}}(I = i) \right\} \nonumber\\
  &\geq c \cdot \left[  \min\left\{ \frac{1}{2}, 1 - D_{I\mid\vec{Y}}(\vec{y})\right\}  - \P_{\vec{y}}(\Omega'^c) \right] \,, \nonumber
\end{align*}
where in the last inequality we used
\begin{equation*}
    \frac{\P_{\vec{y}}(\Omega')}{2} 
      \geq \frac{1}{2} - \P_{\vec{y}}( \Omega'^c) \,,
\end{equation*}
and, using the notation $D_{I\mid\vec{Y}}(\vec{y}) = \max_{i \in [M]} \P_{\vec{y}}(I = i)$,
\begin{equation*}
    \P_{\vec{y}}(\Omega') - D_{I\mid\vec{Y}}(\vec{y})
        = 1 - D_{I\mid\vec{Y}}(\vec{y}) - \P_{\vec{y}}(\Omega'^c) \,.
\end{equation*}
Writing $\delta(\vec{y}) := \P_{\vec{y}}(\Omega'^c)$ with $\expect[\delta(\vec{Y})] = \delta$,
we finally arrive at
\begin{align*}
\expect\bigl[\ell(\vec{Y})\bigr]
  &\geq c \cdot \expect\left[
      \min\left\{ \frac{1}{2}, 1 -  D_{I\mid\vec{Y}}(\vec{Y}) \right\}
      - \delta(\vec{Y}) \right] \\
  &\geq c \cdot \left( \frac{1}{2} \cdot \P\left(D_{I\mid\vec{Y}}(\vec{Y}) \leq \frac{1}{2}\right) - \delta\right) \,.
\end{align*}
This is the assertion.
\end{proof}

\subsection{Separating Gaussian mixtures with linear information}

\begin{lemma} \label{lem:points-in-N01}
    Let $(\vec{v}_i)_{i=1}^M \subset \R^n$ be a family of points
    and $I \sim \Uniform[M]$ a random index.
    Define the random variable $\vec{Y} := \vec{v}_I + \vec{Z}$
    with $\vec{Z} \sim \Normal(\boldsymbol{0},\id_n)$ being a standard Gaussian vector in $\R^n$
    independent of $I$.
    For $r_1 > 0$ consider the set
    \begin{equation*}
        \mathfrak{i} := \left\{i \in [M] \colon \vec{v}_i \in r_1B_2^n\right\} \subseteq [M] \,,
    \end{equation*}
    and write $K := \# \mathfrak{i}$.
    Assume
    \begin{equation*}
      \frac{K}{M} \geq \frac{1}{2}
      \quad\text{and}\quad
      M \geq C \cdot \left(r_1 + 3 \sqrt{n}\right)^n \,,
    \end{equation*}
    where $C = \frac{10\euler^{3/2}}{3\pi} \approx 4.7552$,
    and write $D_{I\mid\vec{Y}}(\vec{y}) = \max_{i \in [M]} \P(I = i \mid \vec{Y} = \vec{y})$ as in Lemma~\ref{lem:truncation}.
    Then we have
    \begin{equation*}
        \P\left(D_{I\mid\vec{Y}}(\vec{Y}) \leq \frac{1}{2}\right)
            \geq \frac{1}{5} \,.
    \end{equation*}
\end{lemma}
\begin{proof}
    The standard Gaussian vector $\vec{Z}$ possesses the probability density
    \begin{equation*}
        p_0(\vec{z}) := \frac{1}{(2\pi)^{n/2}} \exp\left(-\frac{\|\vec{z}\|_2^2}{2}\right)
    \end{equation*}
    which takes its maximum value $\frac{1}{(2\pi)^{n/2}}$ at the origin.
    The joint distribution of~$(I,\vec{Y})$ can thus be described by a family $\{p_i\}_{i=1}^M$ of functions on $\R^d$,
    \begin{equation*}
        p_i(\vec{y}) := \frac{p_0(\vec{y}-\vec{v}_i)}{M} \,,
        \quad \text{giving}\quad
            \P\bigl((I,\vec{Y}) \in A\bigr)
                = \sum_{i=1}^M \int_{\R^n} p_i(\vec{y}) \, \ind_{\{(i,\vec{y}) \in A\}} \dint \vec{y}\,,
    \end{equation*}
    where each of the functions $p_i$ is bounded above by $B := \frac{1}{(2\pi)^{n/2} M}$.
    The probability density of $\vec{Y}$ is given by $p = \sum_{i=1}^M p_i$.
    Furthermore, we can give the conditional probability of $I$ by
    \begin{equation*}
      \P\left(I = i \mid \vec{Y} = \vec{y}\right) = \frac{p_i(\vec{y})}{p(\vec{y})}
    \end{equation*}
    and estimate the probability of interest by
    \begin{equation*}
      \P\left(D_{I\mid\vec{Y}}(\vec{Y}) \leq \frac{1}{2}\right)
        \geq \P\bigl(p(\vec{Y}) > 2B\bigr) \,.
    \end{equation*}
    
    We use the concentration bound~\eqref{eq:|Z|>3sqrt(m)} applied to $\vec{Z}$ with $r_2 = 3\sqrt{n}$,
    \begin{equation*}
        \P\bigl(\|\vec{Z}\|_2 \leq r_2\bigr)
            = \P\bigl(\|\vec{Z}\|_2 \leq 3\sqrt{n}\bigr)
            \geq 1 - \euler^{-2n}
            \geq 1 - e^{-2}\,,
    \end{equation*}
    and find the following estimate:
    \begin{align*}
      \P\bigl(p(\vec{Y}) \geq 2B\bigr) 
            &\geq \P\bigl(\|\vec{Y}\|_2 \leq r_1+r_2\bigr)
                    - \P\bigl(\left(\|\vec{Y}\|_2 \leq r_1+r_2\right) 
                                \wedge \left(p(\vec{Y}) < 2B\right)\bigr)\\
            &\geq \P(I \in \mathfrak{i}) \cdot \P\bigl(\|\vec{Z}\|_2 \leq r_2\bigr)
                    - 2B \cdot (r_1+r_2)^n \cdot \lambda^n\left(B_2^n\right) \,.
    \end{align*}
    Here, in the second step, for the first term we applied the triangle inequality to $\vec{Y} = \vec{v}_I + \vec{Z}$ and used that $\vec{v}_i \leq r_1$ for $i \in \frak{i}$.
    The Lebesgue volume of the $n$-dimensional Euclidean unit ball is given by
    \begin{equation*}
      \lambda^n\left(B_2^n\right) = \frac{n^{n/2}}{\Gamma\left(\frac{n}{2} + 1 \right)} \,.
    \end{equation*}
    We use Stirling's formula to estimate the Gamma function from below:
    \begin{equation*}
        \Gamma\left(\frac{n}{2} + 1\right)
          \geq \sqrt{2\pi}
                \cdot \underbrace{\sqrt{\frac{n}{2} + 1}}_{\geq \sqrt{3/2}}
                \cdot \left(\frac{n}{2\euler}\right)^{n/2}
                \cdot \euler^{-1} 
                \cdot \underbrace{\left(1 + \frac{2}{n}\right)^{n/2}}_{ \in \left[\sqrt{3}, \euler\right]}
          \geq \frac{3\sqrt{\pi}}{\euler}
                \cdot \left(\frac{n}{2\euler}\right)^{n/2} \,.
    \end{equation*}
    With $r_2 = 3\sqrt{n}$ and the value for $B$, we thus obtain
    \begin{align*}
      \P\bigl(p(\vec{Y}) \geq 2B\bigr)
            &\geq \frac{K}{M} \left(1 - \euler^{-2n}\right)
             - \frac{2}{(2\pi)^{n/2} M} \cdot \left(r_1 + 3\sqrt{n}\right)^n \frac{n^{n/2}}{\Gamma\left(\frac{n}{2}+1\right)} \\
            &\geq \frac{K}{M} \left(1 - \euler^{-2n}\right)
            - \frac{2 \euler}{3\sqrt{\pi}} \cdot \left(\frac{\euler}{\pi}\right)^{n/2}
               \cdot \left(r_1 + 3\sqrt{n}\right)^n \cdot \frac{1}{M} \,.
    \end{align*}
    Suppose $\frac{K}{M} \geq \frac{1}{2}$, then $\frac{K}{M} \left(1-\euler^{-2n}\right) \geq \frac{1}{2}\left(1-\euler^{-2}\right) > \frac{2}{5}$.
    Assuming 
    \begin{equation*}
      M \geq \frac{10\euler}{3\sqrt{\pi}} \cdot \left(\frac{\euler}{\pi}\right)^{n/2} \cdot \, \left(r_1 + 3\sqrt{n}\right)^n \,,
    \end{equation*}
    the probability is at least $\frac{1}{5}$ as claimed.
    Since $\frac{\euler}{\pi} < 1$,
    we may simplify
    $\left(\frac{\euler}{\pi}\right)^{n/2} \leq \sqrt{\frac{\euler}{\pi}}$,
    the constant in the lemma can be chosen as
    $C := \frac{10\euler^{3/2}}{3\pi} \approx 4.7552$.
\end{proof}

The following statement generalizes Lemma~\ref{lem:points-in-N01} for arbitrary Gaussian measures with covariance matrix $\Sigma$. Concerning the notation $\sqrt{\Sigma}$, see Section~\ref{sec:gauss} in the appendix.

\begin{lemma} \label{lem:points-in-gauss}
    Let $(\wt{\vec{v}}_i)_{i=1}^M \subset \R^n$ be a family of points
    and $I \sim \Uniform[M]$ a random index.
    Define the random variable $\wt{\vec{Y}} := \wt{\vec{v}}_I + \wt{\vec{Z}}$
    with $\wt{\vec{Z}} \sim \Normal(\boldsymbol{0},\Sigma)$ being a Gaussian vector in $\R^n$
    independent of $I$.
    The covariance matrix $\Sigma$ defines the ellipsoid $E := \sqrt{\Sigma}(B_2^n)$.
    Let $r_1 > 0$ and consider the set
    \begin{equation*}
        \mathfrak{i} := \left\{i \in [M] \colon \wt{\vec{v}}_i \in r_1 E\right\} \subseteq [M] \,,
    \end{equation*}
    we write $K := \# \mathfrak{i}$.
    Assume
    \begin{equation*}
      \frac{K}{M} \geq \frac{1}{2}
      \quad\text{and}\quad
      M \geq C \cdot \left(r_1 + 3 \sqrt{n}\right)^n \,,
    \end{equation*}
    with $C > 0$ as in Lemma~\ref{lem:points-in-N01},
    and write
    $D_{I\mid\wt{\vec{Y}}}(\vec{y}) = \max_{i \in [M]} \P(I = i \mid \wt{\vec{Y}} = \vec{y})$.
    Then we have
    \begin{equation*}
        \P\left(D_{I\mid\wt{\vec{Y}}}(\wt{\vec{Y}}) \leq \frac{1}{2}\right)
            \geq \frac{1}{5} \,.
    \end{equation*}
\end{lemma}
\begin{proof}
  If $\Sigma$ is invertible, consider
  $\vec{Y} := \sqrt{\Sigma^{-1}} \wt{\vec{Y}} = \vec{v}_I + \vec{Z}$
  with $\vec{v}_i := \sqrt{\Sigma^{-1}} \wt{\vec{v}}_i$ and $\vec{Z} := \sqrt{\Sigma^{-1}} \wt{\vec{Z}} \sim \Normal(\vec{0}, \id_n)$.
  Since $\wt{\vec{v}}_i \in r_1 E$ is equivalent to $\vec{v}_i \in r_1 B_2^n$,
  we are precisely in the situation of Lemma~\ref{lem:points-in-N01}, the assertion follows.

  If the covariance matrix $\Sigma$ is singular with rank $k < n$,
  find an orthogonal matrix $Q$ such that $\Sigma = Q D Q^\top$ with
  $D = \diag(\sigma_1,\ldots,\sigma_k,0,\ldots,0)$ where $\sigma_1,\ldots,\sigma_k > 0$.
  We may consider $\widehat{\vec{Y}} := Q^\top \wt{\vec{Y}} = \widehat{\vec{v}}_I + \widehat{\vec{Z}}$
  with $\widehat{\vec{v}}_i := Q^\top \wt{\vec{v}}_i$
  and $\widehat{\vec{Z}} := Q^\top \wt{\vec{Z}} \sim \Normal(\vec{0},D)$,
  where $\wt{\vec{v}}_i \in r_1 E$ is equivalent to
  $\widehat{\vec{v}}_i \in r_1 \widehat{E}$ with
  $\widehat{E} := \sqrt{D} (B_2^n) \subseteq U := \linspan\{\vec{e}_1,\ldots,\vec{e}_k\}$,
  exploiting $Q^\top B_2^n = B_2^n$.
  If we observe $\widehat{\vec{Y}} \in U$
  then we know that $\widehat{\vec{v}}_I \in U$,
  that is, the random index $I$ satisfies
  \begin{equation*}
    I \in \mathfrak{i}'
      := \left\{i \in [M] \colon \widehat{\vec{v}}_i \in U\right\} \subseteq [M] \,.
  \end{equation*}
  The observation $\mathfrak{i} \subseteq \mathfrak{i}'$
  motivates us to consider the conditional setting given $I \in \mathfrak{i}'$.
  There we may resort to a $k$-dimensional standard Gaussian setting using the map
  $T \colon (y_1,\ldots,y_n) \mapsto \left(y_1/\sqrt{\sigma_1},\ldots,y_k/\sqrt{\sigma_k}\right)$,
  thus considering $k$-dimensional vectors
  $\vec{Y}' := T \widehat{\vec{Y}}$ with $\vec{v}_i' := T \widehat{\vec{v}}_i$
  and $\vec{Z}' := T\widehat{\vec{Z}} \sim \Normal(\vec{0},\id_k)$,
  where $\widehat{\vec{v}}_i \in r_1 \widehat{E}$ is equivalent to $\vec{v}_i' \in r_1 B_2^k$.
  This $k$-dimensional setting can be compared with an $n$-dimensional setting
  $\vec{Y} = \vec{v}_{I} + \vec{Z}$
  where $\vec{v}_{I} = (\vec{v}_I';0,\ldots,0) \in \R^n$
  and $\vec{Z} \sim \Normal(\vec{0},\id_n)$ with $Z_j = Z_j'$ for $j=1,\ldots,k$,
  so we can write $\vec{Y} = (\vec{Y'}, \vec{Y''})$.
  Obviously, if $\vec{v}_I' \in r_1 B_2^k$, then also $\vec{v}_I \in r_1 B_2^n$.
  Similarly to the proof of Lemma~\ref{lem:points-in-N01},
  the joint distribution of $(I,\vec{Y}')$ conditioned on the event $\{I \in \mathfrak{i}'\}$
  can be given by a family $\{q_i\}_{i \in \mathfrak{i}'}$ of shifted Gaussian functions
  on~$\R^k$ such that for $i \in \mathfrak{i}'$ we get
  \begin{equation*}
      \P\left(I = i \mid \vec{Y}' = \vec{y}', I \in \mathfrak{i}'\right)
      = \frac{q_i(\vec{y}')}{\sum_{j \in \mathfrak{i}'} q_j(\vec{y}')} \,,
  \end{equation*}
  and also for $(I,\vec{Y})$ there is a family $\{p_i\}_{i \in \mathfrak{i}'}$ of Gaussian functions on~$\R^n$ such that 
  \begin{equation*}
    \P\left(I = i \mid \vec{Y} = \vec{y}, I \in \mathfrak{i}'\right)
      = \frac{p_i(\vec{y})}{\sum_{j \in \mathfrak{i}'} p_j(\vec{y})} \,.
  \end{equation*}
  For $i \in \mathfrak{i}'$ we can factorize
  $p_i(\vec{Y}) = p_i(\vec{Y}',\vec{Y}'') = q_i(\vec{Y'})\,\gamma(\vec{Y''})$
  where $\gamma$ is the density of the $(n-k)$-dimensional standard Gaussian distribution, thus
  \begin{equation*}
      \frac{p_i(\vec{Y})}{\sum_{j \in \mathfrak{i}'} p_j(\vec{Y})}
      =
      \frac{q_i(\vec{Y'}) \, \gamma(\vec{Y''})}{\sum_{j \in \mathfrak{i}'} q_j(\vec{Y'})\,\gamma(\vec{Y''})}
      =
       \frac{q_i(\vec{Y'})}{\sum_{j \in \mathfrak{i}'} q_j(\vec{Y}')} \,.
  \end{equation*}
  With 
  \begin{equation*}
    D_{I|\vec{Y}}'(\vec{y}) := \max_{i \in \mathfrak{i}'} \frac{p_i(\vec{y})}{\sum_{j \in \mathfrak{i}'} p_j(\vec{y})}
  \end{equation*}
  and $M' := \#\mathfrak{i}'$,
  we can hence write
  \begin{align*}
    \P\left(D_{I\mid\wt{\vec{Y}}}(\wt{\vec{Y}}) \leq \frac{1}{2}\right)
      &\geq \P(I \in \mathfrak{i}') 
                \cdot \P\left(D_{I\mid\wt{\vec{Y}}}(\wt{\vec{Y}}) \leq \frac{1}{2}
                            \;\middle|\; I \in \mathfrak{i}' \right) \\
      &= \frac{M'}{M} 
                \cdot \P\left(D_{I|\vec{Y}}'(\vec{Y}) \leq \frac{1}{2}
                            \;\middle|\; I \in \mathfrak{i}' \right)\,.
  \end{align*}
  We use Lemma~\ref{lem:points-in-N01}
  for $\vec{Y}$ conditioned on $\{I \in \mathfrak{i}'\}$
  where $M'$ takes the role of $M$,
  with the small modification that we assume
  $\frac{K}{M'} \geq \frac{M}{2M'}$,
  and $M \geq C \cdot \left(r_1 + 3\sqrt{n}\right)^n$
  instead of a bound on $M'$.
  That way we obtain
  \begin{equation*}
    \P\left(D_{I|\vec{Y}}'(\vec{Y}) \leq \frac{1}{2}
    \;\middle|\; I \in \mathfrak{i}' \right)
      \geq \frac{M}{5M'} \,.
  \end{equation*}
  Combined with the previous estimate, this finishes the proof.
\end{proof}

In our main result, Theorem~\ref{thm:l1->loo,non-ada},
we will consider a collection $(\vec{u}_i)_{i\in [m]}$ of points $\vec{u}_i = \frac{2}{3} \vec{e}_i$
on the coordinate axis, $M = m$.
Thus, given a measurement matrix $N \in \R^{n \times m}$,
we have the images $\vec{v}_i := N\vec{u}_i = \frac{2}{3} \vec{c}_i$
where $\vec{c}_1,\ldots,\vec{c}_m \in \R^n$ denote the colums of $N$.
Crucially, however, we only measure $\vec{Y}_i := N\vec{X}_i = N\vec{u}_i + \sigma N\vec{W}$.
In the following we will answer how badly the columns are separated relative to the noise generated by $N\vec{W}$.

\begin{lemma} \label{lem:cj-in-EJ}
  Let $N \in \R^{n\times m}$ be a matrix with columns $\vec{c}_1,\ldots,\vec{c}_m \in \R^n$.
  For an index set $\mathfrak{j} \subset [m]$
  let $P_{\mathfrak{j}}$ be the projection onto the coordinates $j \in \mathfrak{j}$. 
  Using this, write $\Sigma_{\mathfrak{j}} := N P_{\mathfrak{j}} N^{\top}$
  and define the ellipsoid
  $E_{\mathfrak{j}} := \sqrt{\Sigma_{\mathfrak{j}}}(B_2^n)$.
  For $2n < m$ consider the set $\mathcal{J}_{2n}$ as in \eqref{eq:setsJn'} and
  and let $\mathfrak{J} \sim \Uniform{\mathcal{J}_{2n}}$. Then
  \begin{equation*}
    \P\left(\#\left\{j \in [m] \colon
                     \vec{c}_j \in 2^n E_{\mathfrak{J}}
              \right\}
                \geq \frac{m}{2}
      \right) \geq \frac{1}{2} \,.
  \end{equation*}
\end{lemma}
\begin{proof}
  The proof goes in two stages. In the first stage we fix one
  $\mathfrak{j} \in \mathcal{J}_{2n}$.
  Our aim is to construct a sufficiently large set $\wt{\mathfrak{m}}$
  such that
  \begin{equation} \label{eq:claim-cj-in-2nEj}
    j \in \wt{\mathfrak{m}}
    \quad\Longrightarrow\quad
    \vec{c}_j \in 2^n E_{\mathfrak{j}} \,.
  \end{equation}
  In the course of our construction we iteratively identify
  $n$ most important columns of $N$, described by an index set
  $\mathfrak{j}_n = \{j_1,\ldots,j_n\} \subset \mathfrak{j}$, 
  such that the ellipsoid $E_{\mathfrak{j}_n}$ generated by these $n$~columns
  is contained in $E_{\mathfrak{j}}$.
  The elements of $\wt{\mathfrak{m}}$
  are identified by a shrinkage procedure with a nested sequence of candidate sets,
  \begin{equation*} 
      [m] = \mathfrak{m}_0 \supset \ldots \supset \mathfrak{m}_n \,,
  \end{equation*}
  where for the selected important coordinates $j_1,\ldots,j_n$ we have
  $j_i \in \mathfrak{m}_{i-1} \setminus \mathfrak{m}_i$,
  and finally
  $\wt{\mathfrak{m}} 
  = \{j_1,\ldots,j_n\} \cup \mathfrak{m}_n$.
  The geometric idea behind the shrinkage is
  that after selecting $j_1,\ldots,j_i$
  we define a rectangular set
  $R_i \subset U_i := \linspan\{\vec{c}_{j_1},\ldots,\vec{c}_{j_i}\}$
  such that for $j \in \{j_1,\ldots,j_i\} \cup \mathfrak{m}_i$
  we have $(\id_n - \Pi_i) \vec{c}_j \in R_i$
  where $(\id_n - \Pi_i)$ is the orthogonal projection onto~$U_i \subseteq \R^n$.
  In the end, for $j \in \wt{\mathfrak{m}}$
  our construction guarantees $\vec{c}_j \in R_n$,
  and we also have
  $R_n \subseteq 2^n E_{\mathfrak{j}_n} \subseteq 2^n E_{\mathfrak{j}}$.
  What we show here is therefore actually stronger than
  \eqref{eq:claim-cj-in-2nEj}. 
  
  The variables $k_i = \#\left(\{j_i\} \cup \mathfrak{m}_i\right)$
  keep track of how much the candidate set has shrunk so far.
  In particular, the distribution of $k_n$
  regarded as a random variable $K_n$, depending on a random index set $\mathfrak{J} \sim \Uniform{\mathcal{J}_{2n}}$ 
  instead of a fixed set $\mathfrak{j}$,
  will help us to understand the typical size of the set
  $\wt{\mathfrak{m}}$, namely, $\#\wt{\mathfrak{m}} = k_n + n - 1$.

  Let us now begin with the first stage of the proof
  and fix $\mathfrak{j} \in \mathcal{J}_{2n}$.
  We first describe the construction of $\wt{\mathfrak{m}}$.
  The iteration is initialized by putting
  $\mathfrak{m}_0 := [m]$ with $m_0 := \#\mathfrak{m}_0 = m$,
  and $\Pi_0 := \id_n$.
  For $i = 0,\ldots,n-1$, proceed as follows:
  
  \begin{enumerate}
    \item Sort the elements $j \in \mathfrak{m}_i$
      according to the standard Euclidean norm of $\Pi_i \vec{c}_j$, that is,
      find a bijection $\pi_i \colon [m_i] \to \mathfrak{m}_{i}$ such that
      \begin{equation*}
        \|\Pi_i \vec{c}_{\pi_i(1)}\|_2 \leq \ldots \leq \|\Pi_i \vec{c}_{\pi_i(m_i)}\|_2 \,.
      \end{equation*}
    \item Find the index $j_{i+1}$ from $\mathfrak{j}$ with the largest such norm,
      \begin{align*}
        k_{i+1} &:= \max\{\ell \in [m_i] \colon \pi_i(\ell) \in \mathfrak{j}\}, \\
        j_{i+1} &:= \pi_i(k_{i+1}) \in \mathfrak{j} \cap \mathfrak{m}_i \,, \nonumber \\
        \vec{b}_{i+1} &:= \vec{c}_{j_{i+1}} \,. \nonumber
      \end{align*}
    \item Define
      \begin{align*}
        \mathfrak{m}_{i+1} &:= \left\{\pi_i(\ell) \colon \ell < k_{i+1}\right\}
          \subset \mathfrak{m}_i, \\
        m_{i+1} &:= \#\mathfrak{m}_{i+1} = k_{i+1} - 1 \,,
      \end{align*}
      and let $\Pi_{i+1}$ be the orthogonal projection onto 
      $(\linspan\{\vec{b}_1,\ldots,\vec{b}_{i+1}\})^\bot$.
  \end{enumerate}
  Finally, we put
  \begin{equation} \label{eq:count-wtj=mn+n}
    \mathfrak{j}_i := \{j_1,\ldots,j_i\}, \quad
    \wt{\mathfrak{m}} := \mathfrak{j}_n \cup \mathfrak{m}_n, \quad
    \wt{m} := \#\wt{\mathfrak{m}} = m_n + n = k_n + n - 1 \,.
  \end{equation}
  For this construction, we claim that \eqref{eq:claim-cj-in-2nEj} holds.

  The left singular vectors of $N P_{\mathfrak{j}}$
  are eigenvectors of
  $\Sigma_{\mathfrak{j}} = N P_{\mathfrak{j}} N^{\top}$
  and form a basis of $\R^n$,
  the respective singular values of the former
  being the square roots of the eigenvalues of the latter.
  Further, for $\mathfrak{j}' \subset \mathfrak{j} \subseteq [m]$,
  the coordinate projection $P_{\mathfrak{j}'}$ is just a restriction of $P_{\mathfrak{j}}$.
  We thus have the following inclusion:
  \begin{equation*}
    E_{\mathfrak{j}'}
      = \sqrt{\Sigma_{\mathfrak{j}'}} \left(B_2^n\right)
      = N P_{\mathfrak{j}'} \left(B_2^m\right)
      \subseteq N P_{\mathfrak{j}} \left(B_2^m\right)
      = \sqrt{\Sigma_{\mathfrak{j}}} \left(B_2^n\right)
      = E_{\mathfrak{j}} \,.
  \end{equation*}
  In particular, for $\mathfrak{j} \in \mathcal{J}_{2n}$
  and its accordingly defined subsets $\mathfrak{j}_i$, see above, 
  \begin{equation*}
    E_{\mathfrak{j}_1} \subseteq \ldots \subseteq E_{\mathfrak{j}_n} \subseteq E_{\mathfrak{j}} \,.
  \end{equation*}
  We can compare these ellipsoids to a chain of rectangles $R_1 \subseteq \ldots \subseteq R_n$,
  \begin{equation*}
    R_i := \left\{\alpha_1 \Pi_0\vec{b}_1 + \ldots + \alpha_i \Pi_{i-1} \vec{b}_i \colon
                  \alpha_1,\ldots,\alpha_i \in [-1,1]
           \right\}
      \subset \R^n \,.
  \end{equation*}
  As an intermediate result we aim to show
  \begin{equation} \label{eq:cj-in-R}
    j \in \wt{\mathfrak{m}} = \mathfrak{j}_n \cup \mathfrak{m}_n
    \quad\Longrightarrow\quad
    \vec{c}_j \in R_n \,.
  \end{equation}
   Clearly, the sets $R_i$ and $E_{\mathfrak{j}_i}$ lie in the space
   \begin{equation*}
    U_i := \linspan\{\vec{b}_1,\ldots,\vec{b}_i\}
      = \linspan\{\Pi_0\vec{b}_1,\ldots,\Pi_{i-1}\vec{b}_i\} 
    \,.
  \end{equation*}
  Define
    \begin{equation*}
    n_0 := \min\left(\{n\} \cup \left\{i \in [n-1] \colon \|\Pi_{i} \vec{b}_{i+1}\| = 0 \right\}\right),
  \end{equation*}
  and note that $U_{n_0} = U_n$ and $n_0 = \dim(U_n)$.
  By construction, for all $j \in \mathfrak{m}_i$, $i=1,\ldots,n$, we have
  \begin{equation*}
    |\langle \Pi_{i-1}\vec{b}_i , \vec{c}_j \rangle|
      = |\langle \Pi_{i-1}\vec{b}_i , \Pi_{i-1}\vec{c}_j \rangle|
      \leq \|\Pi_{i-1}\vec{b}_i\| \cdot \|\Pi_{i-1} \vec{c}_j\|
      \leq \|\Pi_{i-1}\vec{b}_i\|^2 \,.
  \end{equation*}
  Hence, using the orthogonal basis $(\Pi_0\vec{b}_1,\ldots,\Pi_{n_0-1}\vec{b}_{n_0})$ of $U_n$, we can write
  \begin{equation*}
    \vec{c}_j
      = \sum_{i=1}^{n_0}
          \underbrace{\frac{\langle \Pi_{i-1}\vec{b}_i , \vec{c}_j \rangle
                          }{\|\Pi_{i-1}\vec{b}_i\|^2}
                      }_{\in [-1,1] \text{ if } j \in \mathfrak{m}_i} 
            \Pi_{i-1}\vec{b}_i \,.
  \end{equation*}
  This shows $\vec{c}_j \in R_{n_0} = R_n$ for $j \in \mathfrak{m}_n \subset \ldots \subset \mathfrak{m}_1$.
  Similarly, $\vec{b}_{i} = \vec{c}_{j_i} \in U_i$, where $j_i \in \mathfrak{m}_{i-1}$, can be given in the form
  \begin{equation} \label{eq:bi-decomp}
    \vec{b}_i
      = \sum_{\ell=1}^{\min\{n_0,i\}}
          \underbrace{\frac{\langle \Pi_{\ell-1}\vec{b}_\ell , \vec{b}_i \rangle
                          }{\|\Pi_{\ell-1}\vec{b}_\ell\|^2}
                      }_{\in [-1,1]} 
            \Pi_{\ell-1}\vec{b}_\ell \,,
  \end{equation}
  implying $\vec{b}_i \in R_i \subseteq R_n$.
  We thus arrive at the intermediate result~\eqref{eq:cj-in-R}.
  
  To prove \eqref{eq:claim-cj-in-2nEj}, it remains to show
  \begin{equation*}
    R_n \subseteq 2^n \, E_{\mathfrak{j}} \,.
  \end{equation*}
  This can be done by inductively proving
  \begin{multline} \label{eq:Ri-inclusion-Ei}
    R_i
      = \left\{\sum_{\ell=1}^{i} \alpha_\ell\Pi_{\ell-1}\vec{b}_\ell \colon
                \max_{\ell \in [i]} |\alpha_\ell| \leq 1 \right\} \\
      \subseteq {\textstyle \sqrt{\frac{4^i - 1}{3}}} \, E_{\mathfrak{j}_i}
      = \left\{\sum_{\ell=1}^{i} z_\ell \vec{b}_\ell \colon 
                \sum_{\ell=1}^{i} z_\ell^2 \leq \frac{4^i - 1}{3} \right\} \,.
  \end{multline}
  For $i=1$, by $\Pi_0 \vec{b}_1 = \vec{b}_1$ and $\frac{4^1 - 1}{3} = 1$,
  the assertion $R_1 \subseteq E_{\mathfrak{j}_1}$ holds with equality.
  For $i=2,\ldots,n$, from \eqref{eq:bi-decomp} we know
  \begin{equation*}
    \vec{b}_i - \Pi_{i-1} \vec{b}_i
      = \sum_{\ell=1}^{\min\{n_0,i-1,n-1\}}
          \underbrace{\frac{\langle \Pi_{\ell-1}\vec{b}_\ell , \vec{b}_i \rangle
                          }{\|\Pi_{\ell-1}\vec{b}_\ell\|^2}
                      }_{\in [-1,1]} 
            \Pi_{\ell-1}\vec{b}_\ell
      \in R_{i-1} \,.
  \end{equation*}
  Hence, an element $\vec{y} \in R_i$ may be written
  with $\vec{y}' \in R_{i-1}$ and $|\alpha_{\ell}|\leq 1$ in the form
  \begin{equation*}
    \vec{y} = \vec{y}' + \alpha_{\ell} \Pi_{i-1} \vec{b}_i
            = \underbrace{\bigl(\vec{y}'
                          - \alpha_{\ell} \, (\vec{b}_i - \Pi_{i-1} \vec{b}_i)\bigr)
                    }_{\in 2 R_{i-1}}
                + \, \alpha_{\ell}\vec{b}_i \,.
  \end{equation*}
  Using the induction assumption
  $2R_{i-1} \subseteq 2\sqrt{\frac{4^{i-1}-1}{3}} \, E_{\mathfrak{j}_{i-1}}$, there is a representation
  \begin{equation*}
    \vec{y} = \sum_{\ell=1}^{i-1} z_{\ell} \vec{b}_\ell + \alpha_i \vec{b}_i
    \quad\text{with}\quad
    \sum_{\ell=1}^{i-1} z_{\ell}^2 + \alpha_i^2
      \leq 4 \cdot \frac{4^{i-1} - 1}{3} + 1
      = \frac{4^i - 1}{3} \,.
  \end{equation*}
  This proves \eqref{eq:Ri-inclusion-Ei}.
  Before we proceed with the second part of the proof,
  note that for all stages $i=1,\ldots,n$ of the construction,
  we have $\#(\mathfrak{m}_i \cap \mathfrak{j}) = 2n-i$.

  For the second part of the proof, instead of a fixed set $\mathfrak{j}$,
  let $\mathfrak{J}$ be drawn uniformly from $\mathcal{J}_{2n}$.
  The construction given in the first part of the proof is now to be understood as random,
  in particular the quantities in~\eqref{eq:count-wtj=mn+n},
  most prominently $\wt{M} = K_n + n - 1$
  which we write with a capital letter to stress the randomness.
  By the geometric considerations above, this random variable $\wt{M}$ serves as a lower bound for
  the quantity of interest:
  \begin{equation*}
    \#\left\{j \in [m] \colon \vec{c}_j \in 2^n E_{\mathfrak{J}} \right\}
        \geq \wt{M} \,.
  \end{equation*}
  
  If the set $\mathfrak{J}$ is random, the value $k_{i+1}$ in step~2 of the construction
  is random as well, so we write $K_{i+1}$ instead to indicate that this is now a random variable.
  In order to understand the distribution of $K_{i+1}$,
  let us first write the random set $\mathfrak{J} = \{K_1',\ldots,K_{2n}'\}$
  with random elements in decaying order $K_1' > \ldots > K_{2n}'$.
  Obviously, $K_1$ and $K_1'$ have the same distribution.
  Further, let $i \in [n-1]$, fix $m-i+1 \geq k_i > k_{i+1} \geq 2n-i$, and, by counting all remaining possibilities for $\mathfrak{J}$, observe that
  \begin{equation*}
    \P(K_{i+1}' = k_{i+1} \mid K_i' = k_i)
      = \P(K_{i+1} = k_{i+1} \mid K_i = k_i) \,.
  \end{equation*}
  In other words,
  the distribution of $K_{i+1}$ does not depend on the particular set $\mathfrak{j}_i$ and the particular ordering $\pi_i$
  the construction would yield for a given realization of~$\mathfrak{J}$.
  Hence, $K_i$ and~$K_i'$ have the same distribution for $i=1,\ldots,n$.
  
  For the special case of even $m = 2k$,
  we can split $[m]$ into two parts $[k]$ and $[m]\setminus[k]$ of equal cardinality.
  By symmetry of the distribution of $\mathfrak{J}$,
  \begin{align*}
      \P(K_n' \geq k+1)
      &= \P\bigl( \#(\mathfrak{J} \cap [m] \setminus [k]) \geq n \bigr) \\
      &= \P\bigl( \#(\mathfrak{J} \cap [k]) \geq n \bigr) \\
      &= \P(K_{n+1}' \leq k) \\
      &= 1 - \P(K_{n+1}' \geq k+1) \\
    [K_{n+1}'<K_n']\qquad
      &\geq 1 - \P(K_n' \geq k+1) \,,
  \end{align*}
  therefore
  \begin{equation*}
    \P(K_n \geq k+1) = \P(K_n' \geq k+1) \geq \frac{1}{2} \,,
  \end{equation*}
  which by $\wt{M} \geq K_n$ shows the assertion.
  For odd $m = 2k+1$ the calculation is similar except that one equality becomes the inequality
  \begin{equation*}
      \P\bigl( \#(\mathfrak{J} \cap [m] \setminus [k]) \geq n \bigr) 
      \geq \P\bigl( \#(\mathfrak{J} \cap [k]) \geq n \bigr) \,,
  \end{equation*}
  further, $k+1 = \lceil\frac{m}{2}\rceil > \frac{m}{2}$ in this case.
\end{proof}

\subsection{Main results for non-adaptive Monte Carlo methods}

The following theorem is the main result of this paper.

\begin{theorem} \label{thm:l1->loo,non-ada}
  For $m \geq C \cdot e^{an^2}$ with suitable constants $C,a > 0$, we have
  \begin{equation*}
    e^{\ran,\nonada}(n,\ell_1^m \embed \ell_\infty^m)
      \geq \frac{1}{3} \left(\frac{1}{20} - e^{-4}\right) =: \eps_0 \approx 0.010561 > 0 
    \,.
  \end{equation*}
  Hence, for $\eps \in (0,\eps_0)$ we have the complexity
  \begin{equation*}
    n^{\ran,\nonada}(\eps, \ell_1^m \embed \ell_\infty^m)
      \geq \sqrt{\frac{\log \frac{m}{C}}{a}} \,.
  \end{equation*}
\end{theorem}
\begin{proof}
  We work with Bakhvalov's technique for the setup of measures $\wt{\mu}$ and $\mu$
  as described in generality in Section~\ref{sec:bakhvalov}
  where $\wt{\mu}$ is the distribution of a random vector $\vec{X} = \vec{u}_I + \sigma \vec{W}$. Here now we consider a simple setting with $M = m$ points $\vec{u}_i = \frac{2}{3}\vec{e}_i$,
  each containing only one non-zero coordinate,
  a random index $I \sim \Uniform[m]$,
  and a Gaussian disturbance $\vec{W}$ into $2n$ random directions.
  The scaling parameter is chosen to be $\sigma := \frac{1}{18n}$, and the truncated measure $\mu$ is supported on the $\ell_1$-balls $F_i$ of radius $r_0 = \frac{1}{3}$ around the points $\vec{u}_i$.
  As before, we truncate with the event $\Omega'$, see~\eqref{eq:Omega'}, and
  via~\eqref{eq:|Z|>3m} applied to the $2n$-dimensional Gaussian distribution
  this yields a truncation probability 
  \begin{equation*}
    \delta = \P(\Omega'^c) = \P\left(\|\vec{W}\|_1 > \frac{r_0}{\sigma} \right)
        < e^{-4n} \leq e^{-4} \approx 0.018316 \,,
  \end{equation*}
  and an $\ell_1$-separation distance
  \begin{equation*}
    \dist_{1}(F_i,F_{i'})
      = \lVert \vec{u}_i - \vec{u}_{i'} \rVert_1 - 2r_0
      = \frac{2}{3} 
  \end{equation*}
  for $i \neq i'$.
  The $\ell_\infty$-separation can be traced back to the distance of the
  coordinate projections of $F_i$ and $F_{i'}$ onto the
  two-dimensional space $\linspan\{\vec{e}_i, \vec{e}_{i'}\}$, namely
  \begin{equation*}
    \dist_{\infty}(F_i,F_{i'})
        = \frac{\dist_{1}(F_i,F_{i'})}{2}
        = \frac{1}{3} =: c > 0 \,.
  \end{equation*}

    Our aim is to provide a lower bound on the performance of randomized algorithms.
    It turns out to be easier to analyse ``extended'' algorithms
    that base their output not only on the linear information $\vec{Y} = N\vec{X}$
    but 
    on the extended information
    $\wt{\vec{Y}} := (\vec{Y},\mathfrak{J})$.
    By Lemma \ref{lem:monotonInfo},
    a lower bound on the error of the extended algorithms $\wt{\phi}(\vec{Y},\mathfrak{J})$
    is also a lower bound for simple algorithms $\phi(\vec{Y})$.
    Lemma \ref{lem:truncation} gives
    \begin{equation} \label{eq:eps0_general}
    \expect\left[ \| \vec{X} - \wt{\phi}(\vec{Y}, \mathfrak{J}) \|_\infty \cdot \ind_{\Omega'} 
            \right] 
        \geq c \cdot \left( \frac{1}{2} 
                        \cdot \P\left( D_{I \mid \vec{Y}, \mathfrak{J}}(\vec{Y},\mathfrak{J}) 
                                            \leq \frac{1}{2} 
                                \right)
                    - \delta
                \right) ,
    \end{equation}
    where $D_{I \mid \vec{Y}, \mathfrak{J}}(\vec{y}, \mathfrak{j})
    = \max_{i \in [M]}\P(I = i \mid \vec{Y} = \vec{y}, \mathfrak{J} = \mathfrak{j})$.
  Conditioned on $\{\mathfrak{J} = \mathfrak{j}\}$,
  the disturbance $\sigma N\vec{W}$ has the covariance matrix $\sigma \Sigma_{\mathfrak{j}}$
  in the information space~$\R^n$, associated with the ellipsoid $\sigma E_{\mathfrak{j}}$,
  see the notation of Lemma~\ref{lem:cj-in-EJ}.
  We apply Lemma~\ref{lem:points-in-gauss}, namely,
  provided $m \geq C \cdot \left(r_1 + 3\sqrt{n}\right)^n$,
  if for $\mathfrak{j} \in \mathcal{J}_{2n}$
  we have
  \begin{equation*}
    \#\left\{j \in [m] \colon N\vec{u}_j \in r_1 \sigma E_{\mathfrak{j}}\right\} \geq \frac{m}{2}
    \,,
  \end{equation*}
  then
  \begin{equation*}
    \P\left( D_{I \mid \vec{Y}, \mathfrak{J}}(\vec{Y},\mathfrak{J}) \leq \frac{1}{2}
            \;\middle|\; \mathfrak{J} = \mathfrak{j} \right)
        \geq \frac{1}{5} \,.
  \end{equation*}
With $N\vec{u}_j = \frac{2}{3} \vec{c}_j$
  and $\sigma = \frac{1}{18n}$,
  Lemma~\ref{lem:cj-in-EJ} shows that this prerequisite is met with probability
  \begin{equation*}
    \P\left(\#\left\{j \in [m] \colon N\vec{u}_j \in r_1 \sigma E_{\mathfrak{J}}\right\} \geq \frac{m}{2}\right) \geq \frac{1}{2}
  \end{equation*}
  if we choose $r_1 = 12n \cdot 2^n$,
  which then leads to
  \begin{equation*}
    \P\left( D_{I \mid \vec{Y}, \mathfrak{J}}(\vec{Y},\mathfrak{J}) \leq \frac{1}{2}\right)
      \geq \frac{1}{5} \cdot \frac{1}{2} = \frac{1}{10} \,.
  \end{equation*}
  Plugging this into~\eqref{eq:eps0_general},
  for any algorithm $\phi \circ N$ we obtain the error bound
  \begin{equation*}
    e^{\mu}(\phi \circ N, \ell_1^m \embed \ell_\infty^m)
      \geq \frac{1}{3} \cdot \left(\frac{1}{2} \cdot \frac{1}{10} - e^{-4}\right)
      =: \eps_0
      \approx 0.010561 \,.
  \end{equation*}
  The condition $m \geq C \cdot \left(r_1 + 3\sqrt{n}\right)^n$
  we need for the application of Lemma~\ref{lem:points-in-gauss} can be
  replaced by a stronger condition, namely
  \begin{equation*}
    m \geq C \cdot e^{an^2} \geq C \cdot \left(12n \cdot 2^n + 3\sqrt{n}\right)^n \,,
  \end{equation*}
  with suitable $a > 0$.
  (A universal choice would be $a = 3\log(3) \approx 3.2951$ with equality for $n=1$, but for $n \geq 18$ we could take $a = 1$. However, no matter how much we restrict $n$, we always have $a > \log(2) \approx 0.6931$.)
  By Bakhvalov's technique, we are done.
\end{proof}

\begin{remark}
    The main result shows a lower bound of order $\sqrt{\log m}$ for the complexity $n$ when aiming for a small error $\eps > 0$.
    In view of the upper bounds we conjecture that a lower bound of order $\log m$ would be optimal.
    The gap seems to be caused by the rather pessimistic inflation factor $2^n$ in Lemma~\ref{lem:cj-in-EJ}. Any improvement that replaces this factor by something polynomial in $n$ would lead to a lower bound of order $\frac{\log m}{\log \log m}$.

    Still, establishing these lower bounds, we have a first proof that the difficulty of the problem $\ell_1^m \embed \ell_\infty^m$ grows with $m$, allowing to establish 
    an approximability result for general linear problems, see Section~\ref{Sec:solvability}.
    Quantitatively, our bounds are yet good enough to compare the non-adaptive setting with adaptive algorithms and prove a result on the possible gaps, see Section~\ref{sec:UB}.
\end{remark}

\begin{corollary}\label{Cor:EmbNonApp}
    For the identical embedding $\ell_1 \embed \ell_{\infty}$ we have
    \begin{equation*}
       \lim_{n \rightarrow \infty} e^{\ran, \nonada}(n,\ell_1 \embed \ell_{\infty})
        \geq \eps_0 > 0 \,.
    \end{equation*}
\end{corollary}

\begin{remark}
    For deterministic algorithms and $n \geq 1$ we have
        \begin{equation*}
            e^{\det}(n, \ell_1 \embed \ell_{\infty}) 
            = 
            \frac{1}{2} \,,
        \end{equation*}
see \cite[Prop~11.11.10]{P80}.
The result in~\cite{P80} is formulated for approximation numbers,
that is, the error of deterministic linear algorithms,
but since the error is measured in the supremum norm,
linear algorithms are optimal in the deterministic setting,
see \cite{CW04} or \cite[Thms~4.5 and~4.8]{NW08}.
We conjecture that randomized methods cannot improve upon the deterministic error in this infinite-dimensional problem,
which would mean that we could replace $\eps_0$ by $\frac{1}{2}$ in Corollary~2.9,
and we speculate that this even holds for adaptive randomized methods.
\end{remark}

\subsection{Impossibility to approximate non-compact operators}
\label{Sec:solvability}

A linear operator $S$ between Banach spaces is called \emph{approximable} with respect to an error notion $e(n,S)$  if
\begin{equation*}
  e(n,S) \xrightarrow[n \to \infty]{} 0 \,.
\end{equation*}

\begin{theorem}\label{Thm:Approxmability}
Let $F,G$ be Banach spaces.
An operator $S:F \rightarrow G$ is approximable with respect to $e^{\ran, \nonada}$ if  and only if    $S$ is compact.
\end{theorem}
To prove Theorem \ref{Thm:Approxmability} we need a result
from \cite{J71} which shows that the identical embedding embedding $\ell_1 \hookrightarrow \ell_{\infty}$ is in a sense a universal non-compact operator.

\begin{proposition}\label{prop:UniversalNonComp}
Let $X,Y$ be Banach spaces and $S:\ell_1 \rightarrow \ell_{\infty}$ be the identical embedding. A bounded operator $T: X \rightarrow Y$ is non-compact if and only if there exist bounded operators $B \colon \ell_1 \to X$, $C\colon Y \to \ell_\infty$ such that $S = CTB$,
that is, the following diagram commutes:
\[
\begin{tikzcd}
X \arrow{r}{T}
& Y \arrow{d}{C} \\
{\ell_1} \arrow{r}{S}
\arrow{u}{B}
& {\ell_\infty}
\end{tikzcd}
\]
\end{proposition}
\begin{proof}[Proof of Theorem \ref{Thm:Approxmability}]
Compact operators are approximable with respect to deterministic algorithms. Since each deterministic algorithm is formally also a randomized algorithm (with trivial dependence on $\omega$) we get that compactness of $T$ is sufficient for approximability.

Now we show that compactness of $T$ is also necessary.
Let $T$ be non-compact and consider the 
factorization $S = CTB$ of the identical embedding $S: \ell_1 \embed \ell_\infty$ from Proposition~\ref{prop:UniversalNonComp}.
We may assume $\|B\| = 1$, otherwise rescale
$B' := \|B\|^{-1} B$ and $C':= \|B\| C$ to get $S = C'TB'$.
Aiming at a contradiction,
assume that there exists a sequence of randomized non-adaptive algorithms
$\widetilde{A}_n = \widetilde{\phi}_n \circ \widetilde{N}_n$ 
for $n \in \N$ which approximates $T$ arbitrarily well, that is, $e^{\ran}\bigl(\wt{A}_n,T\bigr)$ converges to zero.
Define the algorithms $A_n = \phi_n \circ N_n$
with $\phi_n := C \circ \wt{\phi}_n$ and $N_n := \wt{N}_n \circ B$
for approximating $\ell_1 \embed \ell_\infty$.
It is straightforward to check that
\begin{equation*}
  e^{\ran}(A_n, \ell_1 \embed \ell_\infty)
    \leq \|C\| \cdot e^{\ran}\bigl(\wt{A}_n,T\bigr) \xrightarrow[n \to \infty]{} 0 \,.
\end{equation*}
This is a contradiction to Corollary~\ref{Cor:EmbNonApp}.
\end{proof}

\section{Upper bounds for adaptive algorithms} 
\label{sec:UB}

Theorem~\ref{thm:l1->loo,non-ada} is a lower bound for
\emph{non-adaptive} algorithms
applied to the approximation problem $\ell_1^m \embed \ell_\infty^m$.
As mentioned in the introduction, this bound holds for all
embeddings $\ell_p^m \embed \ell_q^m$ with $p,q \in [1,\infty]$,
namely, we have
\begin{equation*}
  e^{\ran,\nonada}(n,\ell_p^m \embed \ell_q^m)
      \geq \eps_0 > 0
\end{equation*}
for $m \geq C e^{an^2}$ with suitable constants $C,a > 0$.
As it turns out, adaptive randomized algorithms can do significantly better
than non-adaptive algorithms in this regime of very large $m$,
at least for embeddings $\ell_1^m \embed \ell_q^m$ with $q \in [2,\infty]$,
as Theorem~\ref{thm:upper} shows.
The theorem is stated for $\ell_1^m \embed \ell_2^m$,
but of course this upper bound also holds for the easier problems
with smaller $\ell_q^m$-norms, that is, $q > 2$,
in particular for $\ell_1^m \embed \ell_\infty^m$.

  The upper bound of Theorem~\ref{thm:upper} is based on an
  adaptive sparse recovery algorithm due to Woodruff et al.~\cite{IPW11,LNW17, LNW18}.
  A vector $\vec{x}$ is called $k$-sparse
  if it has at most $k$ non-zero entries.
  For general $\vec{x} = (x_j)_{j \in [m]} \in \R^m$, Woodruff et al.\ 
  aim to find an approximation $\vec{x}^\ast$
  that is close to the best $k$-term approximation with an $\ell_2/\ell_2$-type of guarantee, namely, for some error demand $\epsilon > 0$, we shall have
  \begin{equation} \label{eq:crite} 
    \|\vec{x} - \vec{x}^\ast\|_2
      \leq \min_{\vec{x}' \colon \text{$k$-sparse}} 
                (1+\epsilon) \|\vec{x} - \vec{x}'\|_2
  \end{equation}
  with high probability. Their algorithm works in two stages.
  The first stage reduces the $k$-sparse recovery problem to a $1$-sparse recovery problem
  by randomly sampling a subset $\mathfrak{m} \subset [m]$ of the big index set 
  where $\#\mathfrak{m} \asymp \frac{\epsilon m}{k}$
  (the authors discuss several adaptive and non-adaptive reduction methodologies).
  With a certain probability,
  the reduced vector $\vec{x}_{\mathfrak{m}} := (x_j)_{j \in \mathfrak{m}}$
  for such a candidate set $\mathfrak{m} \subset [m]$
  contains one ``heavy coordinate'' $j^\ast \in \mathfrak{m}$ in the sense of
  $|x_{j^\ast}| \gg \|\vec{x}_{\mathfrak{m}\setminus\{j^\ast\}}\|_2$.
  The second stage of the algorithm is a shrinkage procedure
  $\mathfrak{m} = \mathfrak{m}_0 \supset \mathfrak{m}_1 \supset \ldots \supset \{j'\}$,
  eventually identifying $j^\ast = j'$ with high probability of success.
  Within each shrinkage step $\mathfrak{m}_i \supset \mathfrak{m}_{i+1}$,
  we compare a random linear measurement of $\vec{x}_{\mathfrak{m}_i}$
  with a slightly disturbed measurement
  to further narrow down the location of $j^\ast$.
  The crucial point is that this shrinking process becomes easier
  the larger the heavy coordinate is
  compared to the remaining coordinates of the candidate set.
  This allows for an accelerated adaptive shrinkage that requires only $\mathcal{O}\left(\log\log\frac{\epsilon m}{k}\right)$
  steps.
  Repeating this $1$-sparse recovery
  for a number of $K \asymp \frac{k}{\epsilon}$ initial candidate sets,
  the algorithm yields an index set $\mathfrak{K} \subset [m]$, $\#\mathfrak{K} \leq K$,
  which hopefully contains (most of) the $k$ largest coordinates.
  The corresponding entries can be measured exactly and the output~$\vec{x}^\ast$
  of the algorithm is defined as
  \begin{equation*}
    x^\ast_j := \begin{cases}
        x_j &\text{if } j \in \mathfrak{K}, \\
        0 &\text{else.}
    \end{cases}
  \end{equation*}
  This also means that, even if the algorithm fails
  to identify the $k$ largest coordinates,
  the error is bounded by the norm of the input vector,
  $\|\vec{x} - \vec{x}^\ast\|_2 \leq \|\vec{x}\|_2$.

\begin{theorem} \label{thm:upper}
  Using adaptive randomized algorithms we obtain the upper bound
  \begin{equation*}  \label{up-bound}
    e^{\ran}(n,\ell_1^m \embed \ell_2^m)
      \preceq \sqrt{\frac{\log \log \frac{m}{n} }{n}} \,,
  \end{equation*} 
  which holds for $n \in \N$ and $m \geq cn$ with a suitable constant $c > 0$.
\end{theorem} 
\begin{proof} 
We use the adaptive randomized algorithm 
of Indyk, Price and Wood\-ruff \cite[Alg~3.2]{IPW11}  
together with the improvements 
of 
Yi Li, Nakos and Wood\-ruff \cite[Thm~11]{LNW17}.
For any $\vec{x} \in \R^m$, one gets a (random) 
$\vec{x}^*$ such that    
\begin{equation*}
    \| \vec{x} - \vec{x}^* \|_2
        \leq 2 \| \vec{x} - \vec{x}_k \|_2
\end{equation*}  
with probability  $1- \exp(-k^{0.99})$, 
where $\vec{x}_k$ is the best approximation of $\vec{x}$ by a $k$-sparse vector.
The algorithm requires
\begin{equation*}
    n \asymp k \,  \log \log \frac{m}{k}  
\end{equation*}
adaptive linear measurements (for $m \geq 16k$).

It is well-known (see, e.g., \cite{CDD09}) 
that for $1 \leq p < q$ and $m>k$ 
we have  
\begin{equation*}
    \| \vec{x} - \vec{x}_k \|_q 
        \leq k^{-r} \| \vec{x} \|_p ,
\end{equation*}
where  $r=1/p-1/q$. 
Here, we consider the case $(p,q)=(1,2)$ and assume $\|\vec{x}\|_1 \leq 1$.
If we aim for an expected error smaller or equal $\eps \in (0,\frac{1}{2})$,
we may apply the algorithm of Woodruff et al.\ for $k = \lceil 16 \eps^{-2}\rceil$
to guarantee an error of at most $2k^{-1/2} \leq \frac{\eps}{2}$ with probability $1 - \exp (-k^{0.99})$.
Here we need
\begin{equation} \label{eq:n=loglogm}
    n \asymp  \eps^{-2} \, \log \log ( m \eps^2)
\end{equation}
adaptive random measurements.
In case the algorithm fails to identify the $k$ most important coordinates,
the error will still be bounded by $1$.
The failure probability is at most $\exp (-k^{0.99}) \leq \frac{\eps}{2}$, 
so the expected error is then bounded by
\begin{equation*}
  e^{\ran}(n,\ell_1^m \embed \ell_2^m)
    \leq \left(1 - \exp (-k^{0.99})\right) \cdot \frac{\eps}{2} + \exp (-k^{0.99}) \cdot 1 
    \leq \eps \,,
\end{equation*}
as claimed.
Finally, we use \eqref{eq:n=loglogm} to derive an error bound in terms of $n$ and $m$.
The restriction $m \geq 16k$ gives $m\eps^2 \geq 256$.
Further, from \eqref{eq:n=loglogm} we know that there exists a constant $c > 0$ such that 
\begin{align}
        n \leq c \, \eps^{-2} \log \log (m\eps^2)
    &\quad\Longleftrightarrow\quad
        \eps^{2} \leq c \, \frac{\log \log (m \eps^2)}{n} \label{eq:eps2}\\
    &\quad\Longleftrightarrow\quad
        \log(m \eps^2) \leq \log \frac{cm}{n} + \log \log \log (m \eps^2) \,. \nonumber
\end{align}
Since $\log \log t < \frac{t}{10}$ for $t > 1$, here with $t = \log(m\eps^2) > \log 256$,
we have
\begin{equation*}
  \log(m \eps^2) \leq \frac{10}{9} \log \frac{cm}{n} \,.
\end{equation*}
This eventually, together with \eqref{eq:eps2}, gives the desired bound,
\begin{equation*}
  \eps \leq \sqrt{c \, \frac{\log \left(\frac{10}{9} \log\frac{cm}{n} \right)}{n}}
    \asymp \sqrt{\frac{\log\log\frac{m}{n}}{n}}\,,
\end{equation*}
where the weak asymptotic equivalence holds for, say, $m \geq 16 \max\{1,c^{-1}\} \, n$.
\end{proof}

\begin{remark}
From the deterministic error asymptotics for the problem $\ell_1^m \embed \ell_2^m$, 
see \eqref{eq:KGG12} in the introduction of this paper,
it follows that if $m \succeq e^n$ then
the error of deterministic algorithms is larger than a constant $\eps_0 >0$ that  
does not depend on $n$ or $m$. 
Therefore, the bound in Theorem~\ref{thm:upper} 
is a great improvement which is possible 
if adaptive randomized algorithms are allowed 
and $m$ is 
large. 

Also the improvement over non-adaptive randomized algorithms 
is substantial. 
Consider the case $m = \bigl\lceil C e^{an^2}\bigr\rceil$ such that 
\begin{equation*}
  e^{\ran,\nonada}(n,\ell_1^m \embed \ell_2^m) \geq e^{\ran,\nonada}(n,\ell_1^m \embed \ell_\infty^m)
    \geq \eps_0 > 0\,.
\end{equation*}
Together with Theorem~\ref{thm:upper} we obtain the following consequence. 
\end{remark}

\begin{corollary} \label{cor:adagap}
For 
$m = \bigl\lceil C e^{an^2}\bigr\rceil$ with the constants $C,a>0$ from Theorem~\ref{thm:l1->loo,non-ada},
we obtain
\begin{equation*}
    \frac{e^{\ran}(n,\ell_1^m \embed \ell_2^m)}
        {e^{\ran,\nonada}(n,\ell_1^m \embed \ell_2^m)}
    \preceq \sqrt{\frac{\log n}{n}}
\end{equation*}
for large $n$.
\end{corollary}

\begin{remark} \label{rem:diag-operator} 
In the corollary above we consider a different problem for every $n$, so an obvious question is: Do we find an infinite-dimensional problem $S$
such that the gap of Corollary~\ref{cor:adagap} can be observed for this single problem? The answer is yes up to logarithmic terms.
Inspired by constructions in \cite[Sec~5]{He24},
we consider the following example
defined with the constants $C,a > 0$ from Theorem~\ref{thm:l1->loo,non-ada}
and some $\alpha > 0$:
\begin{align*}
    &N_k := 2^k,
    \quad m_k := \bigl\lceil C e^{aN_k^2}\bigr\rceil
    \qquad\text{for } k \in \N\,,\\
    &S \colon \ell_1 \to \ell_2, \; \vec{x} \mapsto \vec{z}
    \quad\text{where } z_i := \begin{cases}
        x_i &\text{for } 1 \leq i \leq m_1 \\
        \frac{x_i}{\sqrt{k} \, (\log k)^{1/2 + \alpha}} &\text{for } m_{k-1} < i \leq m_k,\; k\geq 2 \,.
    \end{cases}
\end{align*}
Clearly, we can employ Theorem~\ref{thm:l1->loo,non-ada} to find non-adaptive lower bounds for $k \geq 2$:
\begin{align*}
    e^{\ran,\nonada}(N_k,S)
        &\geq \frac{e^{\ran,\nonada}(N_k,\ell_1^{m_k} \embed \ell_2^{m_k})}{\sqrt{k} \, (\log k)^{1/2 + \alpha}}
        \geq \frac{\eps_0}{\sqrt{k} \, (\log k)^{1/2 + \alpha}} \\
        &\asymp (\log N_k)^{-1/2} \, (\log \log N_k)^{-1/2 - \alpha} \,.
\end{align*}
For the upper bounds we define `splitting' indices $l_j := 2^j$ for $j \in \N$,
and index sets $\frak{m}_1 := [m_{l_1}]$, $\frak{m}_j := [m_{l_j}] \setminus [m_{l_{j-1}}]$ for $j \geq 2$.
Given a total budget of $N_k = 2^k$ pieces of information,
we spend at most $n_j$ information functionals on the approximation
of $\vec{x}_{\frak{m}_j}$ for $j=1,\ldots,k$, where
the individual budgets form a decaying sequence:
\begin{equation*}
    n_j := \left\lfloor c_{\alpha}^{-1} \, j^{-1-2\alpha} \, N_k \right\rfloor
        = \left\lfloor c_{\alpha}^{-1} \, j^{-1-2\alpha} \, 2^k \right\rfloor,\quad
    c_{\alpha} := \sum_{j=1}^\infty j^{-1-2\alpha} \,.
\end{equation*}
If $n_j$ is large enough, we can employ an adaptive algorithm $A_{n_j}$
with budget of at most $n_j$ as in the proof of Theorem~\ref{thm:upper} for approximating $\vec{x}_{\frak{m}_j}$.
Since these algorithms are homogeneous, i.e.\ $A_{n_j}(\lambda \vec{x}_{\frak{m}_j}) = \lambda A_{n_j}(\vec{x}_{\frak{m}_j})$ for scalar factors $\lambda \neq 0$,
we have the following error bounds, valid under a constraint $m \geq cn_j$:
\begin{equation*}
    \expect \|\vec{x}_{\frak{m}_j} - A_{n_j}(\vec{x}_{\frak{m}_j})\|_2
        \preceq \|\vec{x}_{\frak{m}_j}\|_1 \cdot 
            \sqrt{\frac{\log \log \frac{\#\frak{m}_j}{n_j}}{n_j}} \,.
\end{equation*}
With $\log_2 N_k = k$ and $\log \log \frac{\# \frak{m}_j}{n_j} \preceq \log \log (\# \frak{m}_j) \preceq \log N_{l_j} \preceq l_j = 2^j$, and inserting the value of $n_j$,
we can simplify this estimate to
\begin{equation} \label{eq:simplifiedUB}
    \expect \|\vec{x}_{\frak{m}_j} - A_{n_j}(\vec{x}_{\frak{m}_j})\|_2
        \preceq \|\vec{x}_{\frak{m}_j}\|_1
            \cdot \sqrt{\frac{2^j \cdot j^{1+2\alpha}}{2^k}} \,.
\end{equation}
Since for $m \leq n_j$ there is an algorithm $A_{n_j}$ that recovers $\vec{x}_{\frak{m}_j}$ exactly, by adjusting the implicit constant,
\eqref{eq:simplifiedUB} extends to $n_j > m/c$ as well.
If, however, $n_j$ is too small to give an error bound of, say, $\frac{1}{2} \|\vec{x}_{\frak{m}_j}\|_1$, than the zero algorithm $A_{0}(\vec{x}_{\frak{m}_j}) = \vec{0}$
produces an error at most $\|\vec{x}_{\frak{m}_j}\|_2 \leq \|\vec{x}_{\frak{m}_j}\|_1$, so~\eqref{eq:simplifiedUB} holds true for all $j$, even if $n_j$ is small or even zero for certain large $j$.
By approximating $\vec{x}_{\frak{m}_j}$ for $j=1,\ldots,k$, we can find an approximation of $\vec{z}_{\frak{m}_j}$, while approximating the tail $\vec{z}_{\N \setminus [m_{l_k}]}$ with zero.
Applying the triangle inequality to the definition of the Monte Carlo error,
we obtain the following adaptive upper bound:
\begin{align*}
    &e^{\ran}(N_k,S) \\
        &\preceq \sup_{\|\vec{x}\|_1 \leq 1} \left(\frac{\|\vec{x}_{\N \setminus [m_{l_k}]}\|_1}{2^{k/2} \cdot k^{1/2 + \alpha}} + \sum_{j=1}^k \frac{\|\vec{x}_{\frak{m}_j}\|_1}{\max\{2^{(j-1)/2} \cdot (j-1)^{1/2+\alpha},1\}} \cdot \sqrt{\frac{2^j \cdot j^{1+2\alpha}}{2^k}} \right)\,, \\
        &\preceq \max\left\{ \frac{1}{2^{k/2} \cdot k^{1/2 + \alpha}}, 
            \frac{1}{2^{j/2} \cdot j^{1/2 + \alpha}} \cdot \sqrt{\frac{2^j \cdot j^{1+2\alpha}}{2^k}}
        \text{ for } j=1,\ldots,k \right\} \\
        &= 2^{-k/2} = \frac{1}{\sqrt{N_k}} \,.
\end{align*}
Since $N_k$ is chosen in dyadic steps, we find an asymptotic gap for all integers $n \geq 2$:
\begin{equation*}
    \frac{e^{\ran}(n,S)}{e^{\ran,\nonada}(n,S)} \preceq \sqrt{\frac{\log n}{n}} \cdot (\log \log n)^{1/2 + \alpha} \,.
\end{equation*}
\end{remark}

\begin{remark} \label{rem:maximalgap}
Having found such a gap, we may pose the question: How much can be gained in general 
with respect to optimal error bounds if we choose the $n$ 
pieces of linear information
adaptively and/or randomly? 
This question is studied for general linear problems 
$S: F \to G$ between Banach spaces $F$ and $G$
in~\cite{KNU24}. 
Using 
inequalities between $s$-numbers it turns out 
that the maximal gap is of the order $n$ for general spaces 
and of the order $n^{1/2}$ if $F$ or $G$ is a Hilbert space.
In an upcoming paper~\cite{KW24+}, new upper bounds for the adaptive approximation of $\ell_1^m \embed \ell_\infty^m$ will show a gap of order 
$n$ (up to logarithmic terms),
and for $\ell_2^m \embed \ell_\infty^m$ we can also show a gap of order $n^{1/2}$. 

\end{remark} 


\appendix

\section{Gaussian measures} \label{sec:gauss}

A vector $\vec{Z} = (Z_1, \ldots, Z_m)$ in $\R^m$ is called a \emph{standard Gaussian vector}
if its entries are independent standard Gaussian random variables, $Z_j \iid \Normal(0,1)$.
For any given point $\vec{c} \in \R^m$ and any given positive semi-definite matrix $\Sigma \in \R^{m \times m}$
we can construct a Gaussian random vector with \emph{mean}~$\vec{c}$ and \emph{covariance matrix} $\Sigma$ by
\begin{equation*}
  \vec{X} := \vec{c} + \sqrt{\Sigma} \vec{Z} \sim \Normal(\vec{c}, \Sigma) \,.
\end{equation*}
Here, $\sqrt{\Sigma}$ is defined via the functional calculus for self-adjoint matrices,
that is, for diagonal matrices $D = \diag(\sigma_1,\ldots,\sigma_m)$ with non-negative entries we define $\sqrt{D} = \diag(\sqrt{\sigma_1},\ldots,\sqrt{\sigma_m})$,
and for $\Sigma = Q D Q^\top$ with an orthogonal matrix $Q$
we have $\sqrt{\Sigma} = Q \sqrt{D} Q^\top$.
Further, for any matrix $A \in \R^{n \times m}$, the following vector is a Gaussian vector in $\R^n$:
\begin{equation} \label{eq:gaussianProjection}
  \vec{Y} := A\vec{X} \sim \Normal(A\vec{c}, A \Sigma A^\top) \,.
\end{equation}

We cite a very basic version of concentration of Gaussian vectors around the expectation of its norm,
see~\cite[pp.~180/181]{Pis86} or \cite[Lem~2.1.6]{adler2009random}
for a more general statement.
\begin{lemma} \label{lem:concentration}
  Let $\vec{Z}$ be a standard Gaussian vector in $\R^m$ and $\|\cdot\|_{\ast}$ an arbitrary norm on $\R^m$
  with $\ell_2$-Lipschitz constant
  \begin{equation*}
    L_\ast := \sup_{\vec{0} \neq \vec{z} \in \R^m} \frac{\|\vec{z}\|_{\ast}}{\|\vec{z}\|_{2}} \,.
  \end{equation*}
  Then for $t > 0$ we have
  \begin{equation*}
    \P\bigl(\|\vec{Z}\|_{\ast} > \expect \|\vec{Z}\|_{\ast} + t \bigr)
        \leq \exp\left(-\frac{t^2}{2L_\ast^2}\right) \,.
  \end{equation*}
\end{lemma}

The expected $\ell_2$-norm of a standard Gaussian vector $\vec{Z}$ in $\R^m$ is easily estimated from above:
\begin{equation*} 
  \expect \|\vec{Z}\|_{2}
    \leq \sqrt{\expect \|\vec{Z}\|_{2}^2}
    = \sqrt{\sum_{j=1}^m \expect Z_j^2}
    = \sqrt{m} \,.
\end{equation*}
With $L_2 = 1$, for $t = 2\sqrt{m}$ 
we obtain
\begin{equation} \label{eq:|Z|>3sqrt(m)}
  \P\left(\|\vec{Z}\|_{2} > 3\sqrt{m}\right)
    \leq 
        \P\left(\|\vec{Z}\|_{2} > \expect\|\vec{Z}\|_2 + 2\sqrt{m}\right)
    \leq e^{-2m} \,.
\end{equation}
For $\vec{z} \in \R^m$ we have $\|\vec{z}\|_1 \leq \sqrt{m} \cdot \|\vec{z}\|_2$.
Hence,
\begin{equation}\label{eq:|Z|>3m}
  \P\left(\|\vec{Z}\|_{1} > 3m\right)
    \leq \P\left(\sqrt{m} \cdot \|\vec{Z}\|_{2} > 3m\right)
    = \P\left(\|\vec{Z}\|_{2} > 3\sqrt{m}\right)
    \leq e^{-2m} \,.
\end{equation}

\section*{Acknowledgements}

We thank Stefan Heinrich, Aicke Hinrichs, and Thomas K\"uhn for valuable hints.
We also thank the anonymous referees for carefully reading the manuscript.
The initial version of the paper was written in summer 2023
during the first author’s temporary
employment at TU Chemnitz, Faculty of Mathematics.


\bibliographystyle{amsplain}

\bibliography{lit}

\end{document}